\documentclass[12pt,letterpaper]{article}
\usepackage{amsmath,latexsym,amsfonts,amssymb}
\usepackage{wrapfig}
\usepackage{amscd,array,epic,eepic,calc,bbm,float}
\usepackage{ifthen,psfrag,verbatim,epsfig,graphicx,enumerate} 
\usepackage{makeidx}
\usepackage{amsthm}
\usepackage{multirow}
\usepackage{float}

\usepackage{sidecap}
\usepackage{xcolor}
\usepackage{subfigure}

\usepackage{enumitem}

\usepackage{natbib}

\usepackage{lscape}

\usepackage[titletoc,title]{appendix}

\usepackage[subpreambles=true]{standalone}
\usepackage{import}

\usepackage{hyperref}
\hypersetup{hidelinks}

\def\bql{\begin{equation}\label}
\def\eql{\end{equation}\noindent}

\def\ban{\begin{align}}
\def\ean{\end{align}\noindent}

\def\brl{\begin{eqnarray}\label}
\def\erl{\end{eqnarray}\noindent}

\def\bro{\begin{eqnarray*}}
\def\ero{\end{eqnarray*}\noindent}

\def\brr{\begin{array}}
\def\err{\end{array}\noindent}

\def\bdl{\begin{display}\label}
\def\edl{\end{display}\noindent}

\def\bdo{\begin{display}}
\def\edo{\end{display}\noindent}

\def\bth{\begin{theorem}}
\def\eth{\end{theorem}}

\def\bcr{\begin{corollary}}
\def\ecr{\end{corollary}}

\def\bpr{\begin{proposition}}
\def\epr{\end{proposition}}

\def\blm{\begin{lemma}}
\def\elm{\end{lemma}}

\def\bdf{\begin{definition}}
\def\edf{\end{definition}}

\def\bas{\begin{assumptions}}
\def\eas{\end{assumptions}}

\def\bex{\begin{example}\rm}
\def\eex{\end{example}}

\def\bxx{\begin{exercise}\rm}
\def\exx{\end{exercise}}

\def\brm{\begin{remark}\rm}
\def\erm{\end{remark}}

\def\bma{\begin{pmatrix}}
\def\ema{\end{pmatrix}}

\def\bcs{\begin{cases}}
\def\ecs{\end{cases}}

\def\btb{\begin{center}\begin{tabular}}
\def\etb{\end{tabular}\end{center}}

\def\bit{\begin{itemize}}
\def\eit{\end{itemize}}

\def\qed{\quad\hfill\mbox{$\square$}}

\def\a{\alpha}
\def\b{\beta}

\def\g{\gamma}
\def\h{\eta}

\def\j{\psi}

\def\m{\mu}
\def\n{\nu}

\def\q{\theta}
\def\r{\rho}
\def\s{\sigma}

\def\FC{{\cal F}}

\def\pb{{\mathbf p}}

\def\wb{{\mathbf w}}

\def\ind{{\mathbbm{1}}}
\def\thb{\boldsymbol{\theta}}

\DeclareMathOperator{\VaR}{VaR}

\DeclareMathOperator{\CoVaR}{CoVaR}

\def\nbb{{\mathbb N}}
\def\pbb{{\mathbb P}}
\def\rbb{{\mathbb R}}

\def\nf{\infty}

\def\top{\buildrel\pbb\over\rightarrow}

\def\nt{\noindent}

\def\thesection{\arabic{section}}
\def\thesubsection{\arabic{section}.\arabic{subsection}}
\def\theequation{\arabic{section}.\arabic{equation}}

\usepackage[margin = 1in]{geometry} 

\def\theequation{\arabic{section}.\arabic{equation}}

\newtheorem{theorem}{Theorem}[section]
\newtheorem{lemma}[theorem]{Lemma}
\newtheorem{proposition}[theorem]{Proposition}
\newtheorem{corollary}[theorem]{Corollary}
\newtheorem{definition}{Definition}
\newtheorem{example}{Example}
\newtheorem{exercise}{Exercise}
\newtheorem{remark}{Remark}
\newtheorem{assumptions}{Assumptions}[section]

\usepackage{framed}
\usepackage{comment}
\definecolor{shadecolor}{gray}{0.9}
\specialcomment{extra}{\begin{shaded}}{\end{shaded}}

\usepackage{authblk}

\newcommand{\blind}{0}

\makeindex
\usepackage{multibib}

\begin{document}

\def\spacingset#1{\renewcommand{\baselinestretch}%
{#1}\small\normalsize} \spacingset{1}

\if0\blind
{
  \title{\bf An Extreme Value Approach to CoVaR Estimation}
  \author[1]{Natalia Nolde}
\author[2]{Chen Zhou}
\author[1]{Menglin Zhou}

\affil[1]{Department of Statistics, University of British Columbia, Canada}
\affil[2]{Erasmus School of Economics, Erasmus University, The Netherlands}
  \maketitle
} \fi

\if1\blind
{
  \bigskip
  \bigskip
  \bigskip
  \begin{center}
    {\LARGE\bf An Extreme Value Approach to CoVaR Estimation}
\end{center}
  \medskip
} \fi






\bigskip
\begin{abstract}

The global financial crisis of 2007-2009 highlighted the crucial role systemic risk plays in ensuring stability of financial markets. Accurate assessment of systemic risk would enable regulators to introduce suitable policies to mitigate the risk as well as allow individual institutions to monitor their vulnerability to market movements. One popular measure of systemic risk is the conditional value-at-risk (CoVaR), proposed in Adrian and Brunnermeier (2011). We develop a methodology to estimate CoVaR semi-parametrically within the framework of multivariate extreme value theory. According to its definition, CoVaR can be viewed as a high quantile of the conditional distribution of one institution's (or the financial system) potential loss, where the conditioning event corresponds to having large losses in the financial system (or the given financial institution). We relate this conditional distribution to the tail dependence function between the system and the institution, then use parametric modelling of the tail dependence function to address data sparsity in the joint tail regions. We prove consistency of the proposed estimator, and illustrate its performance via simulation studies and a real data example.
\end{abstract}

\nt{\it Key words:} systemic risk; multivariate extreme value theory; tail dependence function; regular variation; heavy tails; method of moments.
\vfill

\newpage
\spacingset{1.8}

\makeatletter
\newcommand{\leqnomode}{\tagsleft@true\let\veqno\@@leqno}
\newcommand{\reqnomode}{\tagsleft@false\let\veqno\@@eqno}
\makeatother
\section{Introduction}

The financial crisis of 2007-2009 revealed an important role systemic risk can play in destabilizing individual markets as well as the global economy. Accurate assessment of systemic risk would enable regulators to identify systemically important financial institutions and introduce suitable policies to mitigate the risk to the system coming from such institutions. For individual institutions, on the other hand, it is their vulnerability to extremal market movements that should be monitored and mitigated. Both of the above situations require a multivariate measure of risk that captures co-movements between a financial system (or market) and individual financial institutions. One popular measure of systemic risk is conditional value-at-risk (CoVaR); \cite{AdrianBrunnermeier2011}. In this paper, we propose a methodology to estimate CoVaR semi-parametrically within the framework of multivariate extreme value theory. This framework is suitable for heavy-tailed financial data, which exhibit what is known as tail or extremal dependence.

Loosely speaking, CoVaR is defined as a high quantile of the conditional distribution of the potential loss of a system proxy such as a market index (or a financial institution) conditional on the event that one institution (or the system) is in distress. The distress event corresponds to an institution experiencing a large loss in excess of a high quantile or the so-called Value-at-Risk (VaR). For a random variable $X$,  VaR at confidence level $1-p$, denoted $\VaR_X(p)$, is defined as $$\VaR_{X}(p) = \inf_x\bigl\{\pbb(X> x)\le p\bigr \},\qquad p\in(0,1).$$
Given two random variables  $X$ and $Y$, the CoVaR at level $1-p$, denoted $\CoVaR_{Y|X}(p)$, is defined as \bql{q1}\pbb\bigl( Y\ge \CoVaR_{Y|X}(p)\mid X\ge \VaR_X(p)\bigr) = p,\qquad p\in(0,1).\eql
This definition of CoVaR is adopted from \cite{Girardi2013} in that the conditioning distress event is given by the exceedence $\bigl\{X\ge \VaR_X(p)\bigr\}$ rather than by $\bigl\{X= \VaR_X(p)\bigr\}$ as originally proposed by \cite{AdrianBrunnermeier2011}. This definition leads to CoVaR being dependence consistent; \cite{MainikSchaanning2014}. The original definition makes it possible to do estimation using quantile regression. In the present paper, we interpret $X$ and $Y$ as losses of a financial institution and a system proxy.


\cite{NoldeZhang2018} proposed a semi-parametric EVT-based approach for CoVaR estimation, which is shown to provide a competitive alternative to flexible fully-parametric methods, such as the one described in \citet{Girardi2013}, while allowing for more relaxed model assumptions. One limitation of the \cite{NoldeZhang2018}'s approach is the requirement of multivariate regular variation on the random vector $(X,Y)$, which, in particular, imposes the restriction of the same tail index for both institutional and system losses. Another limitation is a somewhat restrictive parametric assumption on the extremal dependence structure. In this paper, we propose a more flexible framework requiring only for the system losses to have a regularly varying upper tail. This will, in particular, allow the two components of the underlying random vector to have different tail indices. Furthermore, a greater variety of tail dependence structures can be considered model selection.

In our approach we explore the connection between the definition of CoVaR in~\eqref{q1} and the tail dependence function, assuming existence of the latter. A genuine contribution in this approach is to define and estimate an adjustment factor which captures the impact of the dependence  at extremal levels on the CoVaR. With this adjustment factor, we can express the CoVaR as a quantile at an \emph{adjusted} confidence level of the unconditional distribution, rather than that of the conditional distribution, of system losses.

This approach allows us to break the modelling and estimation procedure into the following three components: (1) estimation of the tail dependence function; (2) computation of the adjustment factor, and (3) univariate high quantiles estimation for the system losses. Steps~(1) and~(3) may be handled in a variety of ways. For step~(1), we suggest a semi-parametric approach as a way to balance model uncertainty and estimation efficiency in view of data sparsity especially in the joint tail. That is, a suitable parametric model is to be chosen from a number of available models for the tail dependence function, with model parameters estimated using, for instance, the moment estimator of \cite{Einmahl_etal2012}. The adjustment factor in step~(2) can then be computed numerically by solving an equation involving the fitted tail dependence function. Finally, for step~(3), we adopt a common assumption of heavy-tailed losses and use an extreme value non-parametric high quantile estimator (\cite{Weissman1978}).

The rest of the paper is organized as follows. Section~\ref{sbg} provides background information on several  probabilistic concepts used in the sequel. In Section~\ref{smethod}, we detail the proposed methodology for CoVaR estimation, prove consistency of the new estimator and illustrate its performance in finite samples using several simulation studies. Section~\ref{sappl} is devoted to an application, in which we apply the proposed CoVaR estimator to time series data of daily losses for several financial institutions in order to quantify their systemic importance in the overall financial market. We also compare performance of our estimator to that of other competing approaches. Conclusions and final discussion are given in Section~\ref{sconc}. The data and R code to reproduce numerical results of the paper are available on GitHub \href{https://github.com/menglinzhou/msCoVaR}{https://github.com/menglinzhou/msCoVaR}.

\section{Background}\label{sbg}

In this section we review several fundamental concepts that will be used in the sequel to develop our methodology for CoVaR estimation.

Regularly varying functions are widely used in extreme value analysis, in particular, to conceptualize heavy-tailed behaviour of random variables and random vectors. A distribution function (df) $F$ on $\rbb$ with an infinite upper endpoint is  \emph{regularly varying} with index $\a>0$, written as $1-F\in\ RV_{-\a}$, if for all $x>0$ $$\lim_{t\to\nf}\dfrac{1-F(tx)}{1-F(t)} = x^{-\a}.$$
Examples of distributions with a regularly varying upper tail include Pareto-like distributions whose upper tail satisfies $$1-F(x)\sim cx^{-\a},\qquad x\to\nf,\qquad \a,c>0. $$
Univariate regular variation also characterizes the maximum domain of attraction of the Fr\'echet  distribution \citep{Gnedenko1943}.

An important aspect of modelling multivariate data is capturing their dependence structure. In the context of multivariate risk measures, the emphasis is on the tail dependence properties of the underlying random vector. There exist a number of analytical tools to describe the extremal dependence structure of a random vector, including the exponent measure and stable tail dependence function; see, e.g., \cite{dHF2006}. In our proposed approach, we make use of the (upper) tail dependence function.

\bdf Consider a random vector $(X,Y)$ with joint df $F$ and continuous margins $F_X,F_Y$. The df~$F$ is said to have the \emph{(upper) tail dependence function} $R$ if for all $x,y>0$, the following limit exists:
\bql{qTDF} \lim_{u\to 0}\frac{\pbb\bigl\{F_X(X)\geq 1-ux, F_Y(Y)\geq 1-uy\bigr\}}{u} =  R(x,y).\eql
\edf

Note that $R(1,1)$ is known in the literature as the upper tail dependence coefficient \citep{Joe1997} and is a popular measure of extremal dependence and risk contagion in finance \citep{QRM}. The case $R(1,1)=0$ is referred to as tail independence, and otherwise we have tail dependence.  

In the proposition below, we summarize several notable properties of the tail dependence function; for details, refer to \cite{dHF2006}, Chapter~6.1.5. 
\bpr\label{pTDF} Let $R$ denote an upper tail dependence function.
\begin{enumerate}[label=\arabic*)] 
\item $R$ is continuous.
\item $R$ is monotonically non-decreasing in each component.
\item $0\leq R(x,y)\leq x\wedge y$.
\item $R$ is homogeneous of order 1: $R(tx,ty) = tR(x,y)$ for any $t>0$.
\end{enumerate}
\epr

\section{Methodology}\label{smethod}

\subsection{Probabilistic framework}\label{sm1}

Let random variables $X$ and $Y$ represent losses of an institution and a system proxy, respectively. The key probabilistic assumptions underlying the proposed methodology for CoVaR estimation include existence of the upper tail dependence function $R$, not identically equal to zero, and that the system proxy random variable $Y$ has a heavy-tailed distribution: $1-F_Y\in RV_{-1/\g}$ for some $\g>0$. Note that no distributional assumptions are made on random variable $X$, losses of an institution.


We next introduce an adjustment factor $\h_p$ defined as
\bql{qeta}
\h_p = \dfrac{\pbb\bigl(Y\ge \CoVaR_{Y|X}(p)\bigr)}{\pbb\bigl(Y\ge \CoVaR_{Y|X}(p)\mid X\ge \VaR_X(p)\bigr)},\qquad p\in(0,1).
\eql
It then follows that $\pbb\bigl(Y\ge \CoVaR_{Y|X}(p)\bigr)= p\h_p$ and hence CoVaR is related to a quantile of the unconditional distribution of $Y$ via
\bql{qeta2}
\CoVaR_{Y|X}(p) = \VaR_Y( p \h_p)
\eql
with the adjusted quantile level $p \h_p$. If $X$ and $Y$ are independent, then CoVaR coincides with the VaR of $Y$ at the same level and $\h_p=1$. In the case of positive quadrant dependence, i.e., when $\pbb(X\ge x,Y\ge y)\ge \pbb(X\ge x)\pbb(Y\ge y)$ for $x,y\in\rbb$, we have $\h_p<1$ and CoVaR is equal to VaR at a higher confidence level determined by $\h_p$.

Going back to the definition of CoVaR in~\eqref{q1} and using \eqref{qeta2}, we have
$$\dfrac{\pbb\bigl(X>\VaR_X(p),\ Y>\VaR_Y(p \h_p)\bigr)}{p} = p. $$
At the same time, since in the risk measurement context the interest lies in small values of risk measure level $p$, the ratio above can be approximated using the tail dependence function in~\eqref{qTDF}:
$$\dfrac{\pbb\bigl(X>\VaR_X(p),\ Y>\VaR_Y(p \h_p)\bigr)}{p}=\dfrac{\pbb\bigl(F_X(X)>1-p,F_Y(Y)>1-p \h_p\bigr)}{p} \approx R(1,\h_p) $$
for values of $p$ sufficiently close to 0. This suggests a possibility of approximating the true adjustment factor $\h_p$ in \eqref{qeta} with an asymptotically determined approximation, denoted $\h_p^*$, which is defined implicitly via
\bql{qeta*} R(1,\h_p^*)=p.\eql
In situations where tail dependence function provides a good approximation of the dependence structure in the tail region, we expect $\h_p$ and $\h_p^*$ to be close. We prove this formally in the supplementary appendix and explore this approximation further in simulation studies in Section~\ref{sim}. 

As the tail dependence function is monotonically non-decreasing in each coordinate (see Proposition~\ref{pTDF}), it follows that 
$R(1,\h)$ is increasing from zero to the value of the tail dependence coefficient $R(1,1)$ for values of $\h$ from zero to one. Hence, a unique solution $\h_p^*$ to equation~\eqref{qeta*} exists provided that $p<R(1,1)$. In applications, $p$ will typically be taken to be small and so generally it will be possible to find the solution as long as dependence is not too close to the tail independence case. 

In Supplementary Appendix S3, we plot $R(1,\h)$ as a function of $\h$ for several models and notice that stronger tail dependence will lead to a smaller value of $\h_p^*$. Thus CoVaR is equivalent to the quantile of the unconditional distribution at a more extreme level $p\h_p^*$; see~\eqref{qeta2}. Regular variation of $1-F_Y$ can be used to give another approximation of CoVaR:
\bql{qRVa}
\CoVaR_{Y|X}(p)\approx \VaR_Y(p\h_p^*)\approx (\h_p^*)^{-\g}\VaR_Y(p)\qquad \text{for } p\ \text{close to zero.}
\eql
The above expression will be used as a basis for constructing an asymptotically motivated estimator of CoVaR. 

\subsection{Estimation}
\label{method:est}
Let $(X_1,Y_1),\ldots,(X_n,Y_n)$ be an i.i.d. sample from a distribution satisfying assumptions stated in Section~\ref{sm1}. Using ideas outlined above, we propose the following estimator of $\CoVaR_{Y|X}(p)$ for small values of $p$:
\bql{qEst}\widehat{\CoVaR}_{Y|X}(p) = (\hat{\eta}_p^*)^{-\hat{\gamma}}\ \widehat{\VaR}_Y(p).
\eql
Tail index $\gamma$ of the distribution of $Y$ assumed to have a regularly varying tail can be estimated using the Hill estimator \citep{Hill1975}: 
        $$\hat{\gamma} = \frac{1}{k_1} \sum_{i=1}^{k_1} \log Y_{n,n-i+1}-\log Y_{n,n-k_1}.$$
The choice of $k_1$ can be automatically decided with a two-step subsample bootstrap method in \citet{Danielsson_etal2001}.

$\widehat{\VaR}_Y(p)$ can be computed using a semi-parametric extreme quantile estimator \citep{Weissman1978}:
\begin{equation}\label{qVaRhat}
\widehat{\VaR}_Y(p) = Y_{n,n-k_2}\left(\frac{k_2}{np}\right)^{\hat{\gamma}},
\end{equation}
where the choice of the sample fraction $k_2$ typically aligns with that of $k_1$.

Finding~$\hat{\h}_p^*$ in~\eqref{qEst} requires an estimate of the tail dependence function. While a number of non-parametric estimators have been proposed in the literature, a parametric assumption on the form of $R$ will lead to efficiency gains in light of data sparsity in the tail region as well as will facilitate computation of an estimate of $\h_p^*$. Assuming a parametric model for the tail dependence function $R(\cdot)=R(\cdot;\thb)$, parameter~$\thb$  can be estimated using one of the methods available in the literature for this estimation problem. \cite{ColesTawn1991} and \cite{Joe_etal1992} apply maximum likelihood method, while \cite{LedfordTawn1996} and \cite{Smith1994} use a censored likelihood approach. \cite{Einmahl_etal2008} point out that these likelihood-based estimation methods require smoothness (or even existence) of the partial derivatives of the tail dependence function. Therefore, as an alternative, they propose an estimator based on the method-of-moments for dimension two, which requires a smaller set of conditions. In the simulation studies and subsequent data analysis, we adopt the method-of-moments (M-estimator) proposed in \cite{Einmahl_etal2008}. This M-estimator has been extended in \cite{Einmahl_etal2012} to be used in arbitrary dimensions and its consistency and asymptotic normality hold under weak conditions. 

We conclude this subsection with the definition of the M-estimator of the tail dependence function.  Let $R_i^X$ and $R_i^Y$ denote, respectively, the rank of $X_i$ among $X_1,...,X_n$ and the rank of $Y_i$ among $Y_1,..,Y_n$ for $i\in \{1,...,n\}$. A nonparametric estimator of the bivariate upper tail dependence function~$R$ is given by:
\begin{equation} \label{qnpR}
    \hat{R}_n(x,y):=\frac{1}{m}\sum_{i=1}^n\ind\left\{R_i^X\geq n+\frac{1}{2}-mx, R_i^Y\geq n+\frac{1}{2}-my \right\},
\end{equation}
where $m=m_n\in\{1,...,n\}$ is an intermediate sequence.

Suppose the function $R$ belongs to some parametric family $\{R(\cdot,\cdot; \thb):\thb \in \Theta\}$, where $\Theta\subset \rbb^p$ $(p\geq 1)$ is the parameter space. Let $g= (g_1,...,g_p)^T: [0,1]^2\to \rbb^p$ be a vector of integrable functions. Define function $\varphi: \Theta\to \rbb^p$ as:
\begin{equation} \label{phi_theta}
    \varphi(\thb):=\int\int_{[0,1]^2}g(x,y)R(x,y;\thb)dxdy.
\end{equation}
Let $\thb_0$ denote the true value of parameter $\thb$. The M-estimator $\hat{\thb}$ of $\thb_0$ is defined as a minimizer of the function (\cite{Einmahl_etal2012})
\begin{equation}
\label{S_mn}
    S_{m,n}(\thb)=\left|\left|\varphi(\thb)-\int\int_{[0,1]^2}g(x,y)\hat{R}_n(x,y)dxdy\right|\right|^2,
\end{equation}
where $||\cdot||$ is the Euclidean norm, and $\hat{R}_n(x,y)$ is the nonparametric estimator of~$R$ in~\eqref{qnpR}. The choice of $m$ and test function $g$ is discussed further in Section~\ref{sim}. 

Once we have $\hat\thb$, $\hat\eta^*_p$ in~\eqref{qEst} can obtained by solving 
\begin{equation} \label{Esteta*}
    R(1,\hat\eta^*_p;\hat\thb)=p.
\end{equation}

\subsection{Consistency}

In this section we state consistency of the proposed CoVaR estimator, and begin by imposing the necessary assumptions to guarantee this result.

Firstly, we present three assumptions that are necessary for consistency of the high quantile estimator in~\eqref{qVaRhat}; see, e.g., Theorem 4.3.8 in \cite{dHF2006}. The first one is the second order condition on the distribution function of $Y$, the second one is for the two intermediate sequences $k_j = k_j(n)$, $j = 1,2$, used in the estimator of VaR in~\eqref{qVaRhat}, and the third one is about the probability level $p = p(n)$.

Denote the quantile function $U_Y=(1/1-F_Y)^{\leftarrow}$, where $\cdot^{\leftarrow}$ is the left-continuous inverse. Then clearly $\VaR_Y(p) =U(1/p)$. 
\begin{enumerate}[label=\textbf{Condition} \Alph{enumi}., ref=Condition \Alph{enumi}, wide=0pt]
\item \label{con:soc} Assume that there exist a constant $\rho< 0$ and an eventually positive or negative function $A(t)$ such that as $t\to\infty$, $A(t)\to 0$ and for all $x>0$,
\begin{equation} \label{eq:soc} \lim_{t\to\infty}\frac{\frac{U_Y(tx)}{U_Y(t)}-x^{\gamma}}{A(t)}=x^{\gamma}\frac{x^{\rho}-1}{\rho}.
\end{equation}
This condition quantifies the speed of convergence in the definition of the heavy-tailedness.

\item \label{con:k} Assume that the intermediate sequences satisfy that as $n\to\infty$
\begin{equation}\label{eq:k assumption}
    k_j\to\infty,\; k_j/n\to 0\text{\ and\ } \sqrt{k_j}A\bigl(n/k_j)\bigr)\to\lambda_j\in\rbb,\quad j=1,2.
\end{equation}
\item \label{con:p} Assume that the probability level $p=p(n)$ is compatible with the intermediate sequences $k_j$, $j=1,2$ as follows:
\begin{equation} \label{eq:condition on p}
   \frac{k_2}{n p}\to\infty \text{\ and \ } \frac{\sqrt{k_1}}{\log(k_2/np)}\to \infty,\quad \text{as } n\to\infty.
\end{equation}
\end{enumerate}

Next, we give the assumptions from Theorem 4.1 in \cite{Einmahl_etal2012}, which guarantee the existence, uniqueness and consistency of M-estimator $\hat\thb$.

\begin{enumerate}[label=\textbf{Condition} \Alph{enumi}., ref=Condition \Alph{enumi}, wide=0pt]
\addtocounter{enumi}{+3}
    \item \label{con:varphi} (i) The function $\varphi$ defined in~\eqref{phi_theta} is homeomorphism from $\Theta \to \rbb^p$ and there exists $\epsilon_0 > 0$ such that the set $\{\thb \in \Theta: ||\thb - \thb_0||\leq \epsilon_0\}$ is closed; (ii) $\thb_0$ is in the interior of the parameter space $\Theta$, $\varphi$ is twice continuously differentiable and the total derivative of $\varphi$ at $\thb_0$ is of full rank.
\end{enumerate}

Last but not least, we impose two conditions on the tail dependence function $R$ in \eqref{qTDF}. The first one aims at controlling the speed of convergence to the limit in~\eqref{qTDF} by a power function and the second one is about the partial derivative of $R$: $R_2(x,y;\thb):=\partial R(x,y;\thb)/\partial y $.
\begin{enumerate}[label=\textbf{Condition} \Alph{enumi}., ref=Condition \Alph{enumi}, wide=0pt]
\addtocounter{enumi}{+4}
    \item \label{con:soc for R} There exists a constant $\tilde\rho>0$, such that as $u\to 0$, uniformly for all $(x,y)\in[0,1]^2\setminus\{(0,0)\}$
\begin{equation}\label{eq:soc for R}
    \frac{1}{u}\pbb\bigl\{F_X(X)>1-ux,F_Y(Y)>1-uy\bigr\}-R(x,y)=O(u^{\tilde\rho}).
\end{equation}
Note that a similar condition has been assumed for the M-estimator for $\thb$, see assumption (C1) in \cite{Einmahl_etal2012}.

\item \label{con:R2} For all $\thb\in \Theta$, the partial derivative $R_2(x,y;\thb)$ is continuous with respect to $y$ in the neighborhood of $(1,0;\thb)$ and $R_2(1,0;\thb)>0$.
\end{enumerate}
Notice that we are going to handle $R(1,\eta^*_p)=p$ as in \eqref{qeta*}. As $p\to 0$, \ref{con:R2} ensures that  $\eta^*_p\to 0$ with the same speed as $p$ for all tail dependence functions in the parametric family.

Consider the CoVaR estimator defined in \eqref{qEst}. Here, $\hat\gamma$ is estimated by the Hill estimator; $\widehat{\VaR}_Y(p)$ is estimated using \eqref{qVaRhat} and $\hat\eta^*_p$ is estimated with \eqref{Esteta*}.
The following theorem shows consistency of the CoVaR estimator defined in \eqref{qEst}. The proof is given in Supplementary Appendix S1.
\begin{theorem} 
\label{main theorem}
    Assume that \ref{con:soc}-\ref{con:R2} hold. In particular, \ref{con:soc for R} holds with $\tilde\rho>1$. Then, as $n\to\infty$,
    $$\frac{\widehat{\CoVaR}_{Y|X}(p)}{\CoVaR_{Y|X}(p)}\top 1.$$
\end{theorem}

\subsection{Simulation studies}\label{sim}

Several simulation studies are conducted in order to assess finite sample properties of the proposed CoVaR estimator. 

To evaluate the CoVaR estimator in~\eqref{qEst} for a given value of risk level $p$, we need to obtain estimates of the tail index parameter $\g$, adjustment factor $\h_p^*$ through the estimate of the parameters of the assumed tail dependence function, and  $\VaR_Y(p)$, the $(1-p)$-quantile of the distribution of $Y$. These three components will naturally all have an impact on the performance  of the CoVaR estimator. Thus, we also report the behaviour of the individual components that comprise the CoVaR estimator to better understand main sources of its bias and estimation uncertainty. 

The assessment is based on 100 Monte Carlo replications at risk level $p=5\%$. To make our estimation procedure automatic for the purpose of simulation studies, we set $k_2 = k_1$ and adopt the bootstrap method in \cite{Danielsson_etal2001} to choose the sample fraction $k_1$ in the Hill estimator of $\g$. 

The samples are simulated from the following five distributions; see Supplementary Appendix S2 for details: 
\begin{enumerate}[label=(\arabic*)]
    \item Bivariate logistic distribution with dependence parameter $\theta\in (0,1]$.
    
    \item Bivariate H\"{u}sler-Reiss (HR) distribution (\cite{HuslerReiss1989}) dependence parameter $\theta >0$.

    \item Bivariate bilogistic distribution(\cite{Smith1990}) parameters $\alpha, \beta\in (0,1)$.

    \item Bivariate asymmetric logistic distribution with dependence parameter $\theta \in (0,1]$ and asymmetry parameters $\phi_1,\phi_2 \in [0,1]$.
    
    \item Standard bivariate $t$ distribution with $\n>0$ degrees of freedom and correlation parameter $\r\in(-1,1)$.

\end{enumerate}

The first four models belong to the class of bivariate extreme value distributions, and hence the tail dependence function for these distributions gives the exact representation of the underlying dependence structure. We let the margins of the first four models to be standard Fr\'{e}chet distribution. The bivariate $t$ distribution is multivariate regularly varying, and consecutively lies in the domain of attraction of a bivariate extreme value distribution with  Fr\'{e}chet margins (see, e.g., Example 5.21 in \cite{Resnick1987}). In this case, the tail dependence function approximates the dependence structure of the underlying distribution in the joint tail region.

The settings of the simulation studies are summarized in Table~\ref{sim:setup}. In particular, for each model used for data generation, we indicate the values of the model parameters, the size of the samples and give specification for the M-estimation of the parameters of the tail dependence function including the value of parameter~$m$ and test function~$g(x,y)$. 

We allow the sample sizes to differ across different models roughly guided by the dimension of parameter space~$\Theta$. For the first three distributions, the sample size is set to $n= 2000$, and we use $n=2500$ for the 3-parameter asymmetric logistic distribution. For the bivariate $t$ distribution, we consider a larger sample size of $n=3000$ as the tail dependence function only gives an approximation of the true dependence structure and estimation of the parameters appears to be more challenging. 

The choice of tuning parameter~$m$ and function~$g(x,y)$ for carrying out M-estimation of the parameters of the tail dependence function is based on the guidance given in \cite{Einmahl_etal2012}. As the M-estimator is sensitive to the value of $m$, its choice in simulation studies is made on the basis of the behaviour of the bias and root mean squared error (RMSE). Larger values of $m$ lead to increase in the absolute value of the bias. However, in most cases, $m$ can be chosen so as to minimize the RMSE. These ``optimal" values of $m$, provided in Table~\ref{sim:setup}, are subsequently employed in CoVaR estimation. The choice of function~$g(x,y)$ does not exert much influence on M-estimator and can be taken to have a simple form to facilitate computation. Our specific choices are presented in Table~\ref{sim:setup}.

\begin{table}[H]
\footnotesize
\centering
\caption{Set-up of the simulation studies including distributions for data generation, their parameter values, sample size, and specifications for M-estimation of the parameters of the tail dependence function, including parameter~$m$ and test function~$g(x,y)$.}
\vspace{24pt}
\begin{tabular*}{1\textwidth}{@{\extracolsep{\fill}} c|c |c |c |c}\hline\hline
Model & Parameters & $n$ & $m$ & $g(x,y)$\\
\hline\hline
Logistic & $\q=0.6$ & 2000 & 180 & $g(x,y)=1$\\\hline
HR & $\q=2.5$ & 2000 & 280 & $g(x,y)=x$\\\hline
Bilogistic & $(\a,\b) = (0.4,0.7)$ & 2000 &180 & $g(x,y)=(1,x)^T$\\\hline
Asymmetric logistic &  $(\q,\j_1,\j_2)=(0.6,0.5,0.8)$ & 2500 &180 & $g(x,y)=(1, x, 2x+2y)^T$\\\hline
Bivariate t & $(\n,\r)=(5,0.6)$ & 3000 & 100 & $g(x,y) = (x, x+y)^T$\\\hline\hline
\end{tabular*}
\label{sim:setup}
\end{table}

The summary statistics of the CoVaR estimates are reported in Table~\ref{sim:summary}. The first row gives the true values of $\CoVaR_{Y|X}(p)$ under the various considered models, computed by finding the quantile of the conditional distribution, which is given as the solution to equation $h(y)=p^2$ with
\begin{equation}
    h(y) = \int_{\{(u,v)\in\rbb^2: u>\VaR_X(p), v > y\}}f(u,v)dudv,
\end{equation}
where $\VaR_X(p)$ is the value of the $(1-p)$-quantile of the distribution of $X$, and $f(x,y)$ is the joint density function of random vector~$(X,Y)$. As both mean and median of the estimates exceed the true CoVaR value, these results reveal the tendency of the proposed estimator to overestimate the true value. However, approximate 95\% confidence intervals based on asymptotic normality of the sample mean do cover the true values. From the applied perspective, the proposed estimation procedure offers a conservative estimator of systemic risk as measured by CoVaR. 

\begin{table}[H]
\footnotesize
\centering
\caption{Summary statistics of CoVaR estimates at level $p=0.05$ for simulation settings specified in Table~\ref{sim:setup}. The first row gives the true value of CoVaR under each model. The bottom three panels correspond to CoVaR estimates when the specified component is held at the true value rather than being estimated.}
\vspace{24pt}
\begin{tabular*}{1\textwidth}{@{\extracolsep{\fill}} c|ccccc}\hline\hline
       & logistic & HR       & bilogistic & asymmetric logistic & bivariate t      \\ \hline\hline
$\CoVaR_{Y|X}(0.05)$ & 367.31 & 399.48 & 341.52   & 281.49            & 4.42 \\  
\hline
   \multicolumn{6}{c}{Full estimator} \\   \hline
Mean  & 446.34 & 463.40 & 460.94   & 327.75            & 4.50 \\ 
Median & 425.50 & 456.33 & 434.50   & 314.42            & 4.48 \\ 
Standard deviation     & 127.92 & 130.75 & 149.86   & 83.40             & 0.55 \\ \hline
 \multicolumn{6}{c}{True $\gamma$} \\   \hline
Mean  & 325.14 & 350.26 & 352.88   & 253.01            & 4.13 \\ 
Median & 324.57 & 345.26 & 351.78   & 251.79            & 4.10 \\ 
Standard deviation     & 29.44 & 31.81 & 34.44   & 26.58             & 0.46 \\ 
\hline
 \multicolumn{6}{c}{True $\eta^*_p$} \\   \hline
Mean  & 436.37 & 463.38 & 459.72   & 311.55            & 4.26 \\ 
Median & 415.40 & 456.32 & 433.12   & 312.67            & 4.28 \\ 
Standard deviation     & 117.98 & 130.73 & 148.33  & 69.19             & 0.48 \\ \hline 
\multicolumn{6}{c}{True $\eta_p$} \\   \hline
Mean  & 439.39 & 463.39 & 394.40   & 320.69            & 4.47 \\ 
Median & 418.26 & 456.33 & 372.40   & 321.78            & 4.48 \\ 
Standard deviation     & 118.98 & 130.73& 122.51   & 71.77             & 0.50 \\ \hline\hline
\end{tabular*}
\label{sim:summary}
\end{table}

We have further investigated the sources of bias and variance of the CoVaR estimator. In particular, the panel in Table~\ref{sim:summary} labelled ``True $\g$" presents summary statistics of CoVaR estimates with the tail index kept at its true value, rather than being estimated. Here we observe a substantial bias and variance reduction. The bottom two panels show results based on the true values of the approximate adjustment factor $\h_p^*$ and the exact adjustment factor $\h_p$. The value of $\eta_p^*$ is computed by solving equation $R(1,\eta_p^*;\thb) = p$, while $\h_p$ is evaluated from the following expression:
\begin{equation}
    \eta_p = \frac{\pbb\left\{Y > \CoVaR_{Y|X}(p)\right\}}{p}.
\end{equation}
Little difference can be attributed to the use of $\h_p$ instead of $\h_p^*$. And while, as expected, the use of true values or either $\h_p$ or $\h_p^*$ reduces both the bias and standard deviation of the CoVaR estimator, the reductions are modest compared to those of the tail index parameter. Hence, we can conclude that it is the estimator of $\g$ that is largely responsible for the bias and variability of the proposed CoVaR estimator.

A more detailed view of the performance of the CoVaR estimator as well as its components is given by sampling density plots in Figure~\ref{sim:plot}. In the rightmost panels, we indicate the first approximation of CoVaR, $\CoVaR^*_{Y|X}(p):=\VaR_Y(p\h_p^*)$, based on the true values of the adjustment factor $\h_p^*$ and ($1-p\h_p^*$)-quantile of the distribution of~$Y$; see~\eqref{qRVa}. Note that,  apart from being used to explore performance of estimators, distances between $\eta_p$ and $\eta_p^*$ (see 
the second panels), and $\CoVaR_{Y|X}(p)$ and $\CoVaR^*_{Y|X}(p)$ can also be used to indicate how well the upper tail dependence function approximates the conditional tail probability when $p$ is small. Based on the sampling densities, we observe the presence of a small positive bias in the Hill estimator of the tail index $\g$. This could potentially be remedied by the use of a bias-corrected estimator. The estimator of the adjustment factor $\h_p^*$ tends to perform fairly well across all models, with only a modest negative bias visible for the logistic and bivariate t distributions. Furthermore, the difference between the exact value of the adjustment factor $\h_p$ and its approximation $\h_p^*$ is quite small relative to the sampling variability of the estimator under all models but the bilogistic distribution (see Figure~\ref{sim:plot}(c)). In the latter case, the visual distance may be attributed to fairly small variability of the estimator. Comparison of CoVaR estimates based on $\h_p$ and $\h_p^*$ reveals only a minor impact of the observed discrepancy in $\h_p$ and $\h_p^*$ values. It is interesting to note that, in the case of the asymmetric logistic and bivariate t distributions, $\hat\h_p^*$ appears to be more accurate for the exact value $\h_p$ rather than $\h_p^*$, although the overall influence of $\hat\h_p^*$ on the estimation of CoVaR appears to be quite small. The final component of the CoVaR estimator is the estimator of $\VaR_Y(p)$. As typical for a high quantile estimator, the sampling density displays pronounced skewness to the right. However, the mode tends to coincide well with the true value of the quantile. Here the proposed CoVaR estimator naturally inherits properties of a high quantile estimator but with variability further amplified due to extrapolation to an even more extreme quantile level.

\begin{figure}[H]
  \centering 
  \subfigure[Logistic model]{
    \includegraphics[width=0.98\linewidth]{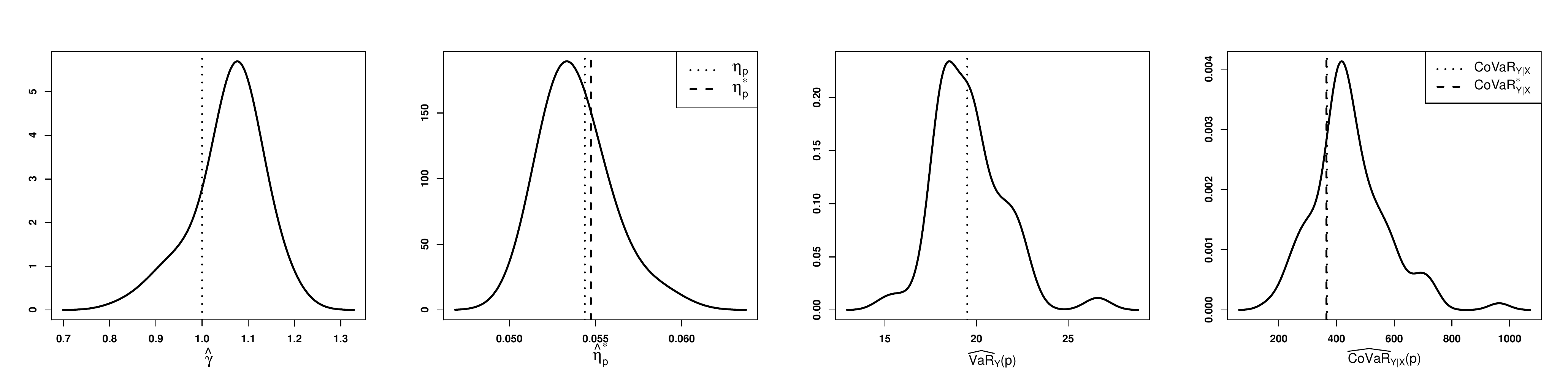}}
    \end{figure}
    \begin{figure}[H]
    \subfigure[HR model]{
    \includegraphics[width=0.98\linewidth]{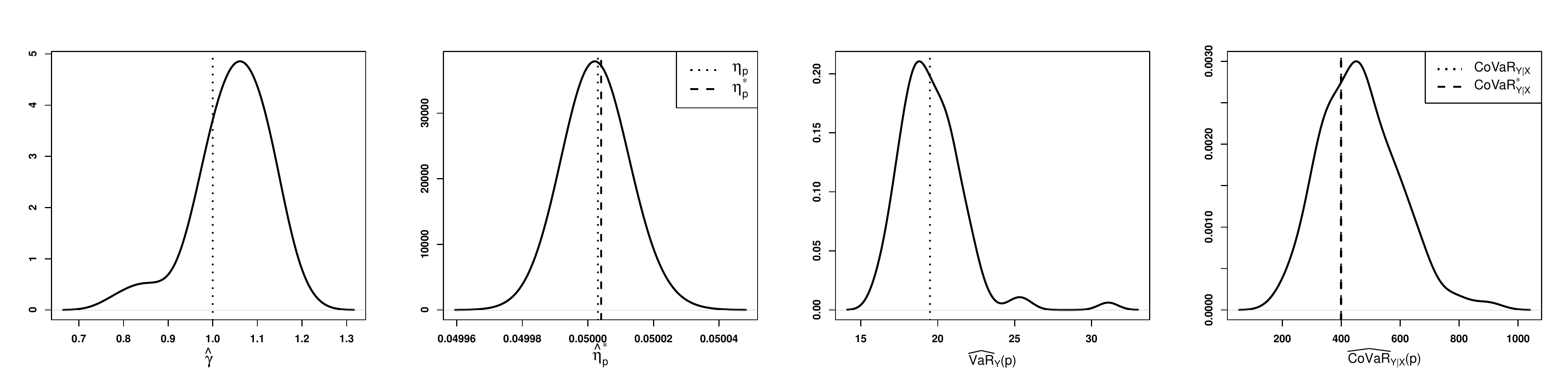}}
    \end{figure}
    \begin{figure}[H]
    \subfigure[Bilogistic model]{
    \includegraphics[width=0.98\linewidth]{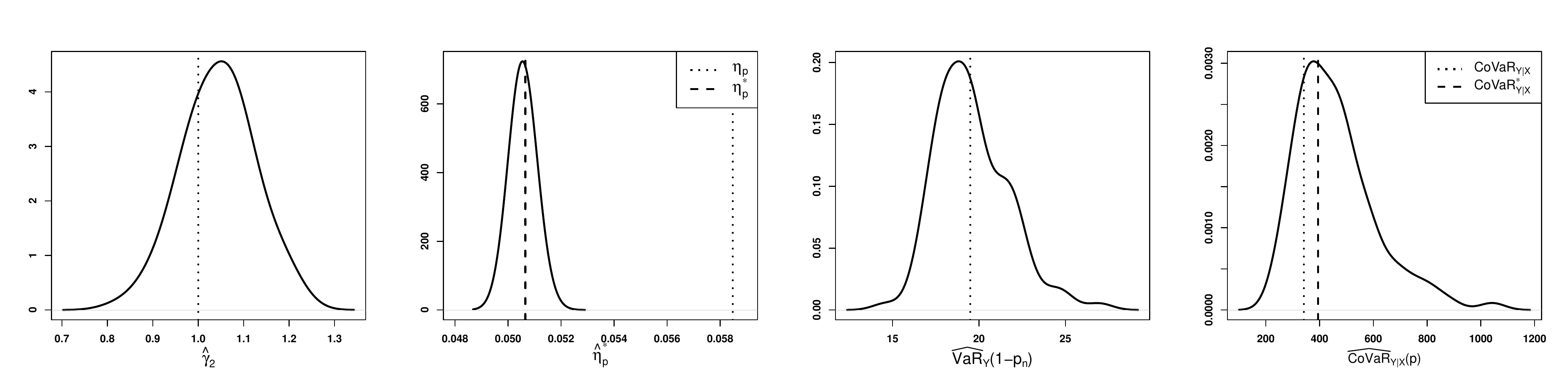}}
    \end{figure}
    \begin{figure}[H]
    \subfigure[Asymmetric logistic model]{
    \includegraphics[width=0.98\linewidth]{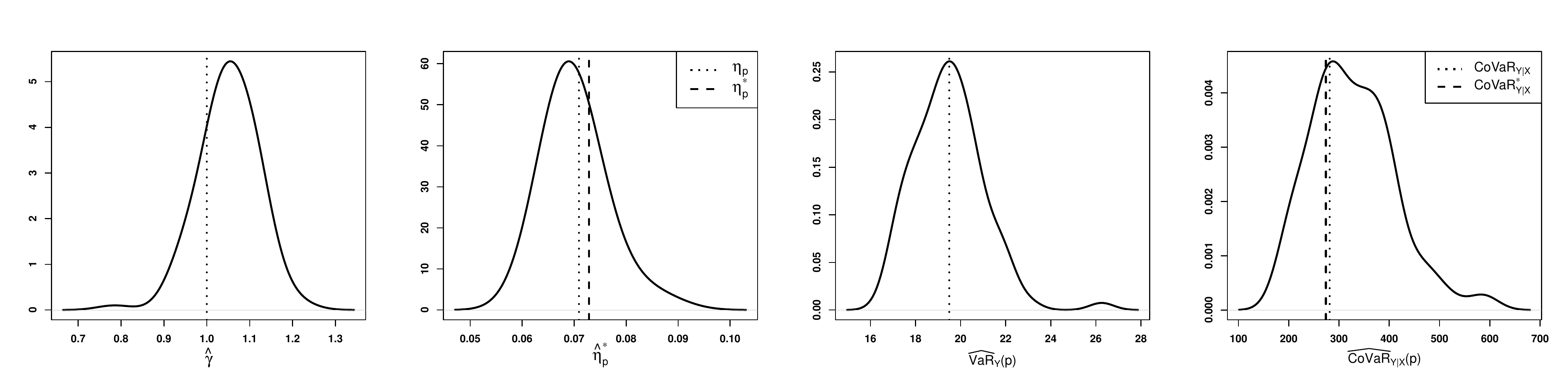}}
    \end{figure}
    \begin{figure}[H]
    \subfigure[bivariate t model]{
    \includegraphics[width=0.98\linewidth]{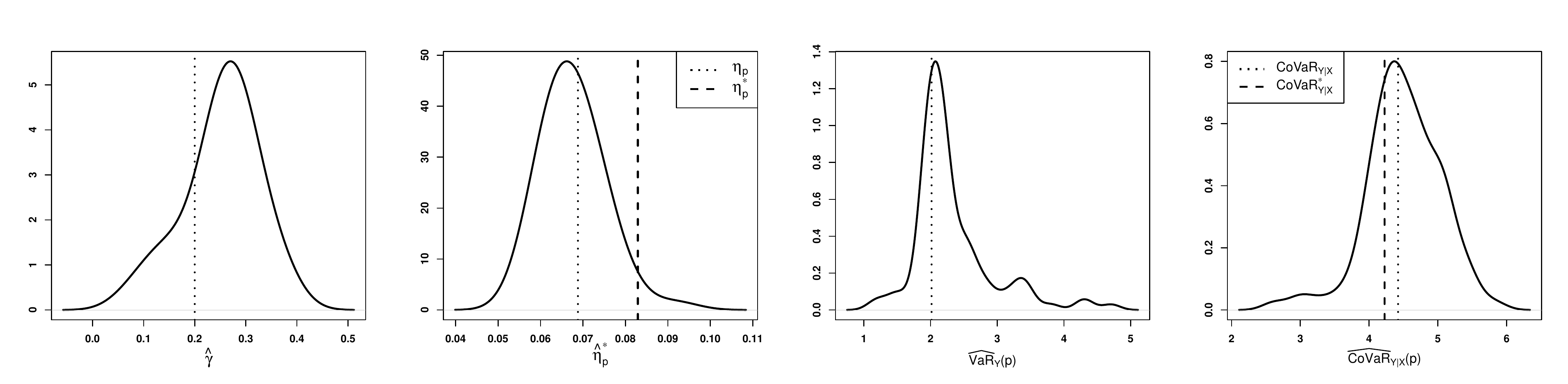}}
  \caption{The sampling densities of estimates of $\gamma$,  $\eta_p^*$, $\VaR_Y(p)$; see Table~\ref{sim:setup} for simulation settings.}
  \label{sim:plot}
\end{figure}

\subsection{Extension of CoVaR definition to different risk levels}

It is possible to extend the definition of CoVaR to allow for different risk levels in the conditioning event $\{X\ge\VaR_X(p_1)\}$ and in the event of $Y$ exceeding the CoVaR; see \citet{MainikSchaanning2014}. In this case, we define CoVaR at level~$\pb = (p_1,p_2)$, denoted $\CoVaR_{Y|X}(p_1,p_2)$, as the $(1-p_2)$-quantile of the conditional loss distribution:
\bql{qCoVaRe}
\pbb\bigl(Y\ge \CoVaR_{Y|X}(p_1,p_2)\mid X\ge\VaR_X(p_1)\bigr) = 1-p_2,\qquad p_1,p_2\in(0,1).
\eql
This extended definition of CoVaR is useful in applications as it allows to consider a less extreme risk level for CoVaR in comparison to the conditioning event (with $p_2$ larger than $p_1$), which results in more observations being available for model validation and backtesting. 

To modify the estimator of CoVaR for the extended definition, we follow steps analogous to those presented in Section~\ref{sm1}. In particular, we note that
$$\CoVaR_{Y|X}(p_1,p_2) = \VaR_Y( p_2\h_\pb)\quad \text{with}\quad \h_\pb=\dfrac{\pbb\bigl(Y\ge \CoVaR_{Y|X}(p_1,p_2)\bigr)}{\pbb\bigl(Y\ge \CoVaR_{Y|X}(p_1,p_2)\mid X\ge\VaR_X(p_1)\bigr)}. $$
Hence, the following equality holds based on the definition of CoVaR:
$$\dfrac{\pbb\bigl(X\ge\VaR_X(p_1),\ Y\ge \CoVaR_{Y|X}(p_1,p_2)\bigr)}{p_1} = p_2, $$
and applying the marginal probability integral transforms leads to the approximation:
$$\dfrac{\pbb\Big(F_X(X)\ge 1- p_1,\ F_Y(Y)\ge 1-\h_\pb\frac{p_2}{p_1}p_1\Big)}{p_1}\approx R\Big(1,\h_\pb\frac{p_2}{p_1}\Big)\qquad \text{for } p_1\approx 0.$$
We can then define the approximate adjustment factor $\h_\pb^*$ via equation
\bql{qeta*2}R\Big(1,\h_\pb^*\frac{p_2}{p_1}\Big) = p_2. \eql
These steps then suggest the following estimator for CoVaR at level $\pb = (p_1,p_2)$:
\begin{equation}
    \label{covar_hat}
    \widehat{\CoVaR}_{Y|X}(p_1,p_2) = \widehat{\VaR}_Y(1-p_2)(\hat{\eta}_\pb^*)^{-\hat{\gamma}} = Y_{n,n-k_2} \Big(\dfrac{k_2}{np_2}\Big)^{\hat{\gamma}} (\hat{\eta}_\pb^*)^{-\hat{\gamma}},
\end{equation}
where $\hat{\gamma}$ is the Hill estimator of $\gamma$ with sample fraction $k_1$. Estimation of $\h_\pb^*$ is carried out in the same way as discussed earlier via parametric estimation of the tail dependence function and subsequently solving equation~\eqref{qeta*2}. 

\section{Application}\label{sappl}

In this section, we illustrate how the CoVaR estimation methodology presented in the previous section can be utilized to produce dynamic CoVaR forecasts using financial time series. In addition, we compare the proposed methodology with the fully-parametric method of \cite{Girardi2013} and the EVT-based method of \cite{NoldeZhang2018}.

\subsection{Data description}\label{data}

In our application, we consider 14 financial institutions studied in \cite{Acharya_etal2017} with a market capitalization in excess of 5 billion USD as of the end of June 2007, including AFLAC INC (AFL), AMERICAN INTERNATIONAL GROUP INC (AIG), ALLSTATE CORP (ALL), BANK OF AMERICA CORP (BAC), HUMANA INC (HUM), J P MORGAN CHASE \& CO (JPM), LINCOLN NATIONAL CORP (LNC), M B I A INC (MBI), PROGRESSIVE CORP OH (PGR), U S A EDUCATION INC (SLM), TRAVELERS COMPANIES INC (TRV), UNUMPROVIDENT CORP (UNM), WELLS FARGO \& CO NEW (WFC), WASHINGTON MUTUAL INC (WM). The S\&P 500 index (GSPC) is used as a system proxy. The sample period is from January 1, 2000 to December 30, 2021, consisting of 5535 daily closing price records for each time series. The daily losses (\%) were calculated as negative log returns. Figure~\ref{fig:time series} gives the time series plots of daily losses for one of the institutions (AFL) and the GSPC index, both displaying a
typical behaviour for financial time series including periods of volatility clustering such as the one during the global financial crisis of 2007--2009.

In selection of financial institutions for the data analysis, we considered the length of their available data records as well as compliance with model assumptions for the proposed method. In addition, we chose to include institutions for which tail index estimates differed from that of the S\&P 500 index.

\begin{figure}[ht]
  \centering 
   \subfigure[AFL]{
    \includegraphics[width=0.45\linewidth]{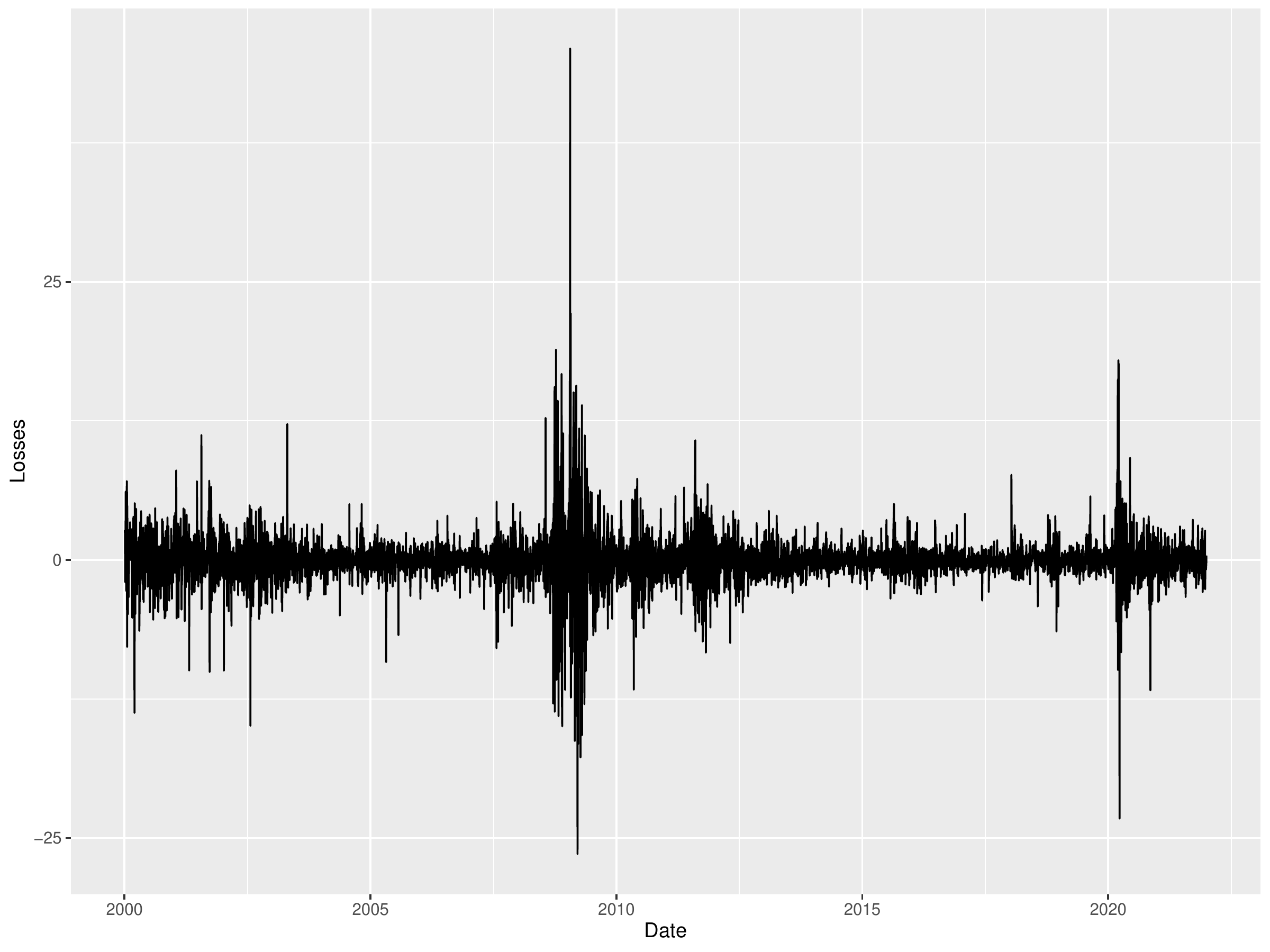}}
    \subfigure[S\&P 500]{
    \includegraphics[width=0.45\linewidth]{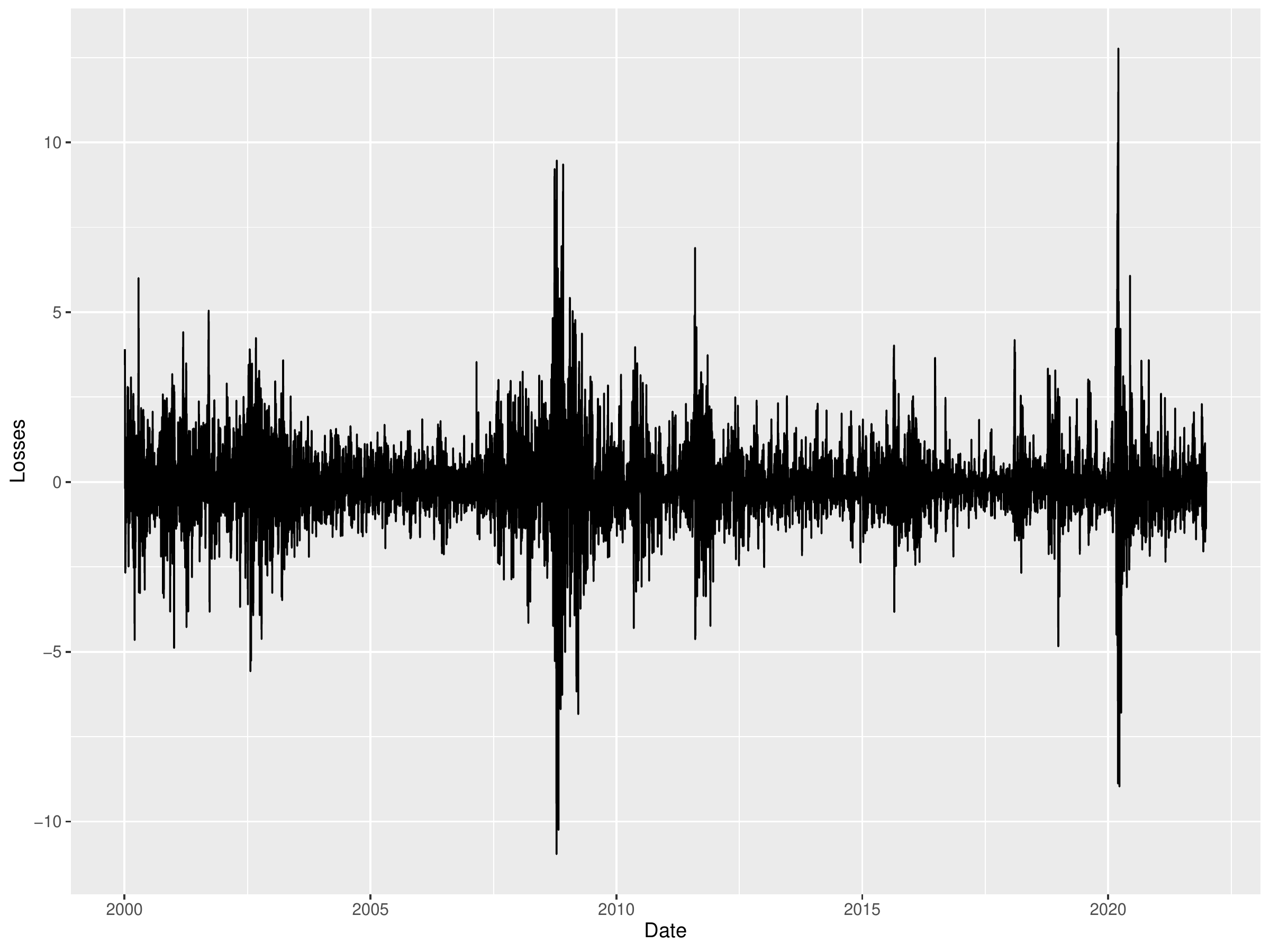}}
      \caption{Time series plots of daily losses for AFLAC INC (AFL) and the S\&P 500 index.}
  \label{fig:time series}
\end{figure}

\subsection{CoVaR estimation and forecasting in the dynamic setting}\label{sappl2}

The methodology outlined in Section~\ref{smethod}, developed under the premise of i.i.d. observations, is not directly suited to produce dynamic estimates and forecasts of CoVaR for financial time series, known to possess serial dependence and display volatility clustering. One way to address this issue is by combining a GARCH-type model for capturing the evolution of the conditional mean and variance of the underlying stochastic process with an EVT-based static treatment of the i.i.d. innovations; see, e.g. \cite{McNeilFrey2000}. This leads to a two-stage procedure in which first an ARMA-GARCH process is fitted to the returns (or losses) data, assuming a parametric model for innovations, followed by applying an EVT-based estimation procedure to the sample of realized residuals.

Let $\{X_t^i\}_{t\in \nbb}$ and $\{X_t^s\}_{t\in \nbb}$ denote time series of losses for an institution (or company) and a system proxy (market index) adapted to the filtrations $\FC^i=\{\FC_t^i\}_{t\in \nbb}$ and $\FC^s=\{\FC_t^s\}_{t\in \nbb}$, respectively. To produce dynamic forecasts, we next define conditional versions of risk measures at time $t$ given information in the series up to time $t-1$.
The (conditional) VaR at confidence level $p_1\in (0,1)$ for $X_t^i$ given information on the institution's losses up to time $t-1$, denoted $\VaR^i_{t}(p_1)$,  is defined as the $(1-p_1)$-quantile of the distribution of~$X_t^i$ conditional on $\FC_{t-1}^i$:
$$\pbb\bigl(X_t^i\geq \VaR^i_{t}(p_1)\mid \FC_{t-1}^i \bigr) = 1 - p_1,$$
and $\CoVaR_{t}^{s|i}(p_1,p_2)$ is defined as the $(1-p_2)$-quantile of the conditional loss distribution given information on losses up to time $t-1$ for both the institution and the system proxy:
\begin{equation}
\label{dynamic_CoVaR}
    \pbb\bigl(X_t^s \geq \CoVaR_{t}^{s|i}(p_1,p_2) | X_t^i\geq \VaR^i_{t}(p_1);\ \FC_{t-1}^i,\   \FC_{t-1}^s\bigr) = 1 - p_2.
\end{equation}
The details of the two-stage procedure for estimating $\CoVaR_{t}^{s|i}(p_1,p_2)$ are provided in the Supplementary Appendix S4.

\subsection{In-sample analysis}
\label{insample}
In this section, we perform stationary CoVaR estimation at risk level $\pb = (0.02,0.05)$ on the basis of realized residuals from the AR(1)-GARCH(1,1) filter. Performance of the CoVaR estimator in~\eqref{covar_hat} under several parametric models for the tail dependence function is assessed via the unconditional coverage test with performance comparisons made using average quantile scores; see \cite{Banulescu_etal2020} and \cite{FisslerHoga2021} for details on backtesting of CoVaR.

Before applying the proposed methodology to the data, we first check appropriateness of several underlying assumptions, including positive quadrant dependence and tail dependence. Based on the distance-based tests given in \cite{Tang_etal2019}, we find that for the 14 institutions considered in Section~\ref{data} the null hypothesis of independence against the alternative of positive quadrant dependence is rejected with p-values below $0.001$. In order to validate the tail dependence assumption,  we estimate the upper tail dependence coefficient, $R(1,1)$, via the method proposed in \cite{Lee_etal2018}, which uses extrapolation of the tail-weighted measures $\zeta_{\alpha}$ over a sequence of $\alpha$ values. Standard errors for these estimates are obtained using a 5000-fold bootstrapping scheme. The estimated values and corresponding standard errors are given in Table~\ref{lambda_est}. For all firms, the upper tail dependence coefficient estimates are significantly greater than zero, thus validating the assumption of tail dependence between an institution and the market index.

\begin{table}[H]
\footnotesize
\centering
\caption{Estimates of the upper tail dependence coefficient $R(1,1)$ with the corresponding standard errors~(SE).}
\vspace{24pt}
\begin{tabular}{cccccccccccc}
\hline
Index  & $\hat{R}(1,1)$ & SE & Index  & $\hat{R}(1,1)$ & SE & Index  & $\hat{R}(1,1)$ & SE & Index  & $\hat{R}(1,1)$ & SE \\ \hline
AFL & 0.342  & 0.024 & AIG & 0.362 & 0.023 & ALL & 0.296 & 0.023 & BAC & 0.430 & 0.023 \\
HUM & 0.162 & 0.024 & JPM & 0.447 & 0.023 & LNC & 0.440 & 0.023 & MBI & 0.248 & 0.023 \\
PGR & 0.278 & 0.024 & SLM  & 0.259 & 0.024 & TRV & 0.320 & 0.023 & UNM & 0.377 & 0.024\\
WFC & 0.389 & 0.023 & WM & 0.290 & 0.022 \\ \hline
\end{tabular}
\label{lambda_est}
\end{table}

Estimation of VaR and CoVaR requires choosing suitable values for sample fractions $k_1$ and $k_2$. While it is common to take $k_2 = k_1$ and select a value with the two-step subsample bootstrap algorithm such as in \citet{Danielsson_etal2001}, we have found that this procedure leads to very low values of $k_1$ and $k_2$ for the considered datasets. As a result, in the present data analysis, we allow the two values to be different. We first select a value of $k_1$ for the tail index using the Hill plot. Then we perform a sensitivity analysis of VaR estimates to values of $k_2$ and select $k_2$ from a stable region. 

Note that the CoVaR estimator in~\eqref{covar_hat} does not involve explicit computation of the VaR estimates for each institution in the conditioning event. However, for the purpose of backtesting, these estimates are needed, in particular, in order to carry out the unconditional coverage test. The VaR estimates for the institutions are obtained using the extreme quantile estimator in~\eqref{qVaRhat}. The selected $k_1$ and resulting estimates of the tail index $\g$ are reported in Table~\ref{gamma_est}. Note that the estimate for the S\&P 500 index (GSPC) is lower than for the institutions, suggesting a lighter tail likely due to the effect of diversification.  

\begin{table}[H]
\centering
\footnotesize
\caption{Estimates of the tail index $\g$.}
\vspace{12pt}
\begin{tabular}{cccccccccccc}
\hline
Index  & $k_1$ & $\hat{\g}$ & Index   & $k_1$ & $\hat{\g}$ & Index   & $k_1$ & $\hat{\g}$ & Index & $k_1$ & $\hat{\g}$ \\ \hline
AFL & 250  & 0.353 & AIG & 250 & 0.349 & ALL & 300 & 0.336 & BAC & 250 & 0.308 \\
HUM & 250 & 0.356 & JPM & 250 & 0.307 & LNC & 270 & 0.314 & MBI & 300 & 0.362 \\
PGR & 300 & 0.342 & SLM  & 250 & 0.330 & TRV & 250 & 0.345 & UNM & 250 & 0.369\\
WFC & 250 & 0.276 & WM & 250 & 0.359 & GSPC & 200 & 0.257 \\ \hline
\end{tabular}
\label{gamma_est}
\end{table}

Using these estimated values, the plots of $k_2$ versus corresponding VaR estimates for all institutions and the index guide selection of a value for $k_2$. Figure~\ref{fig:VaRplots} show the plots for AFL and GSPC; plots for other institutions are similar to that of AFL. We observe that the curve for AFL seems to be stable for values around 220 to 280, while the curve for GSPC is stable from 150 to 210. Thus we select $k_2 = 250$ for all institutions and $k_2 = 200$ for GSPC.

\begin{figure}[ht]
  \centering
   \subfigure[AFL]{
    \includegraphics[width=0.35\linewidth]{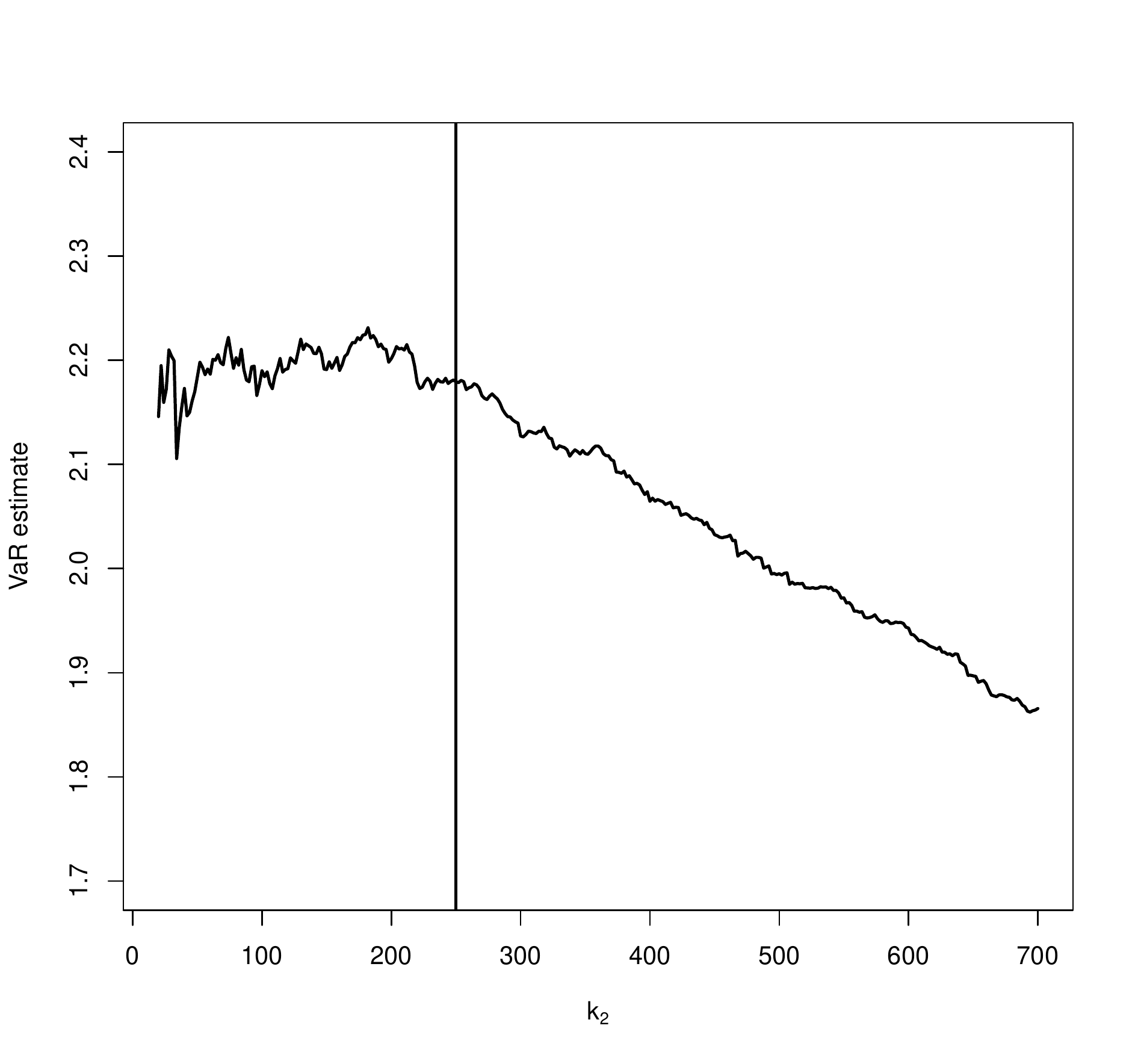}}
   \subfigure[GSPC]{
    \includegraphics[width=0.35\linewidth]{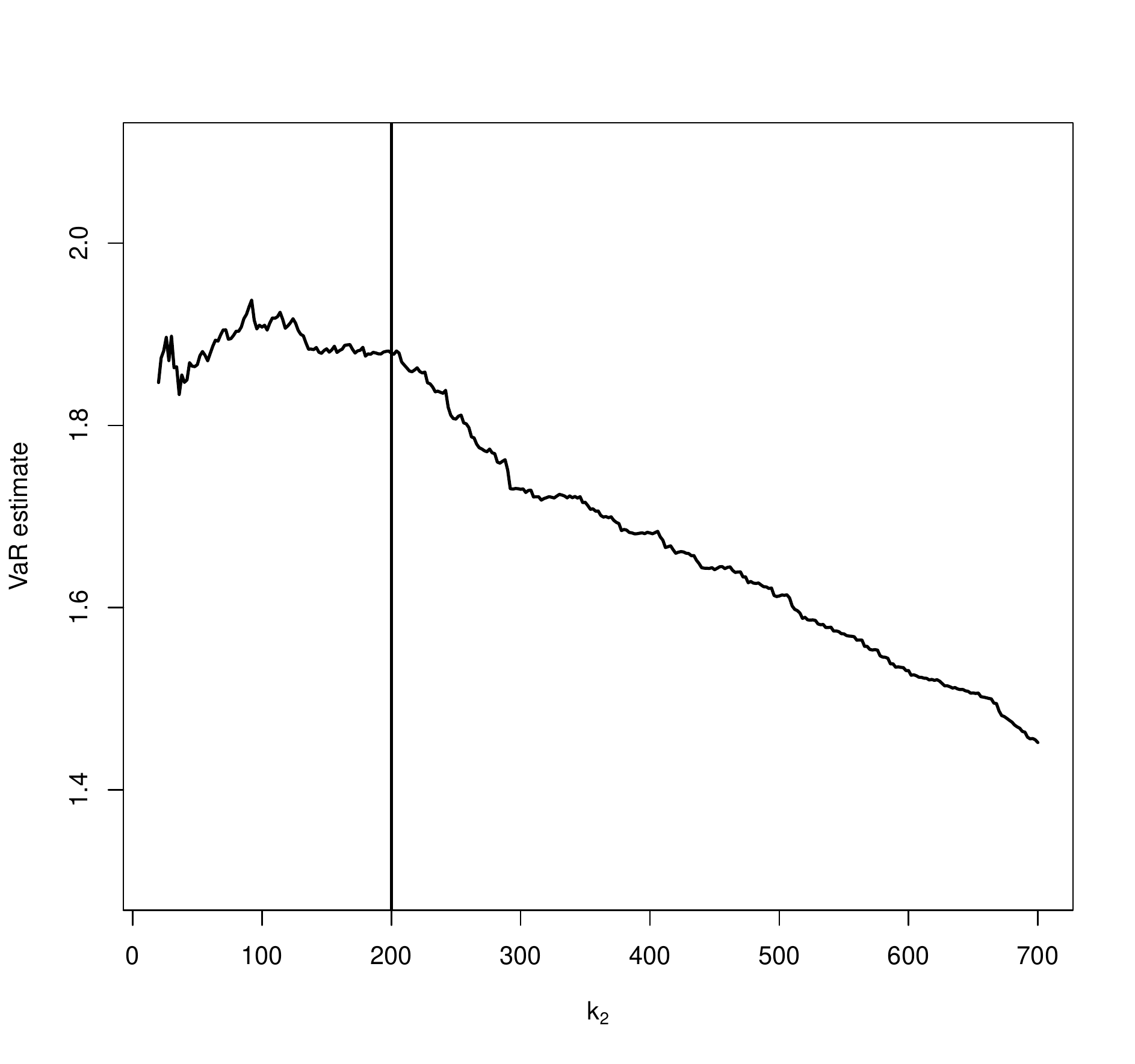}}
\caption{Estimates of VaR as a function of $k_2$ for realized residuals. The risk level for AFL is 0.02 and for GSPC it is 0.05. The vertical lines indicate the selected value of $k_2$. }
  \label{fig:VaRplots}
\end{figure}

Results of the unconditional coverage tests are summarized in Table~\ref{inno_uncon}\footnote{Tail dependence models include logistic (Log),  H\"{u}sler-Reiss (HR), bilogistic (Bilog), asymmetric logistic (Alog) and that of the bivariate t distribution (t); see the supplementary appendix for model specifications. $E_n/e_n$ is the observed/nominal number of exceedances of the VaR estimate, and $E_n^b/e_n^b$ is the observed/nominal number of joint exceedances of VaR and CoVaR estimates. ``FP" corresponds to the fully parametric method of \cite{Girardi2013}; ``EVT-NZ" is for the EVT-based method in \cite{NoldeZhang2018}. }. For all institutions, the VaR estimates pass the unconditional coverage test at 5\% significance level. The CoVaR estimates under the logistic, HR and bilogistic models for the tail dependence function seem to be overestimated, as the values of $e_n^b$ are always greater than those of $E_n^b$, especially under the HR model, where the unconditional coverage test is rejected at 5\% significance level for 5 institutions. On the other hand, the proposed estimator with the symmetric logistic and bivariate t tail dependence models, the fully parametric method as well as the EVT-NZ method provide much better calibrated estimates of CoVaR, passing all of the unconditional coverage tests.

To make further comparisons across the considered methods, we then summarize the average quantile scores for each CoVaR estimator in Table~\ref{ave_score} (see the top panel), using the classical 1-homogeneous scoring function for $(1-p_2)$-quantile: $S(r,x) = \bigl(p_2 - \ind\{x>r\}\bigr)r + \ind\{x>r\}x$ with $r$ denoting the estimate or forecast and $x$ the observation. When comparing within the proposed methodology, the asymmetric logistic and t models for the tail dependence function lead to a superior performance relative to the other three models for institutions other than TRV. And the HR model shows the worst performance for all the 14 institutions. These observations are consistent with the results of calibration reported in Tables~\ref{inno_uncon}. Furthermore, when compared with the fully parametric method in \cite{Girardi2013}, which assumes the bivariate skew-t distribution, and the EVT-based method in \cite{NoldeZhang2018}, for 12 out of 14 companies, the proposed methodology leads to a better performance in terms of accuracy of CoVaR estimates in the in-sample analysis. 

We note that, while conceptually CoVaR can be backtested in the same way as VaR, conditioning on an institution's losses being above its VaR estimate or forecast creates a practical difficulty to obtaining conclusive results when performing comparative backtesting due to substantial reduction in the size of the testing data set. Figure~\ref{figTLM}(a) shows traffic light matrix (see, e.g., \cite{NoldeZiegel2017}) for comparative backtests for institution MBI at 10\% test confidence level. For this institution, the proposed method with the tail dependence functions based on the bivariate t distribution and asymmetric logistic distribution has a significantly better performance than any of the other considered methods. For many of the remaining institutions, traffic light matrices tend to contain many yellow cells making comparisons statistically inconclusive but indicating that the proposed method is not worse than the other competing approaches.

If one is interested in a method that performs best across all institutions, one can combine \emph{normalized} scores.  Pooling information in such way leads to more conclusive results. Figure~\ref{figTLM}(b) shows traffic light matrices for comparative backtests based on the normalized average scores combined across all institutions. It confirms that the proposed method under the bivariate t model for the tail dependence function is significantly superior to all of the other methods in the in-sample analysis.

\subsection{Dynamic CoVaR forecasting}
\label{dynamic}
In this section, we perform a dynamic analysis to assess accuracy of out-of-sample forecasts of CoVaR at risk level $\pb = (0.02,0.05)$ for the times series described in Section~\ref{data}. We use a rolling window of 3000 data points to estimate model parameters and produce one-day ahead CoVaR forecasts according to the two-stage procedure detailed in the Supplementary Appendix S4. However, to reduce computational time, CoVaR and VaR estimates based on the samples of realized innovations are updated only every 50 observations. The resulting average quantile scores under different tail dependence models for the proposed estimator and two other competing approaches are presented in Table~\ref{ave_score} (see the bottom panel).

Some of the conclusions here are similar to those for the in-sample analysis. In particular, when applying the proposed method, the tail dependence functions from the asymmetric logistic and t distribution are preferred. They together yield superior performance relative to the other tail dependence functions for all but one institution (HUM). Comparing the proposed method with the fully parametric approach of \cite{Girardi2013} and an EVT-based estimator in \cite{NoldeZhang2018}, the new estimator produces the lowest average score for 10 out of 14 institutions. This slightly worse performance of the new estimator in the dynamic setting can be explained by reduction in the size of the estimation window, which affects accuracy of the M-estimator of the tail dependence function parameters. Compared to the results for the in-sample analysis in the top panel, the performance of the bivariate t model for tail dependence decreases dramatically for institutions HUM and WM, which can be attributed to computational complexity of the M-estimator due to the lack of a closed form expression for the tail dependence function. However, across all considered institutions, 
the proposed method with the tail dependence function based on the asymmetric logistic distribution provides a superior forecasting performance as indicated by the traffic light matrix in Figure~\ref{figTLM}(c).

\begin{figure}[ht] 
\centering
\subfigure[MBI]{\includegraphics[width=0.32\linewidth]{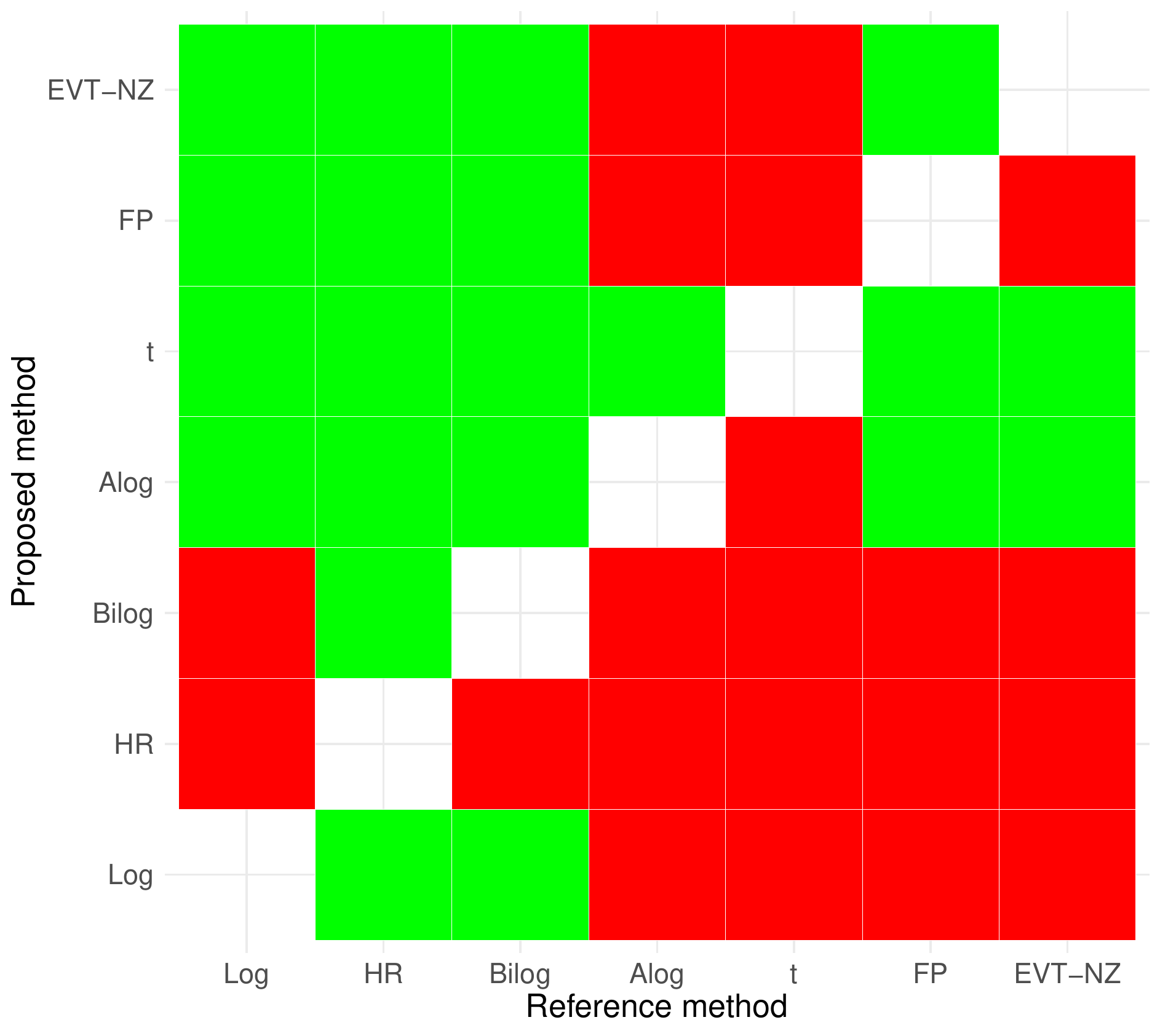}}
\subfigure[In-sample]{\includegraphics[width=0.32\linewidth]{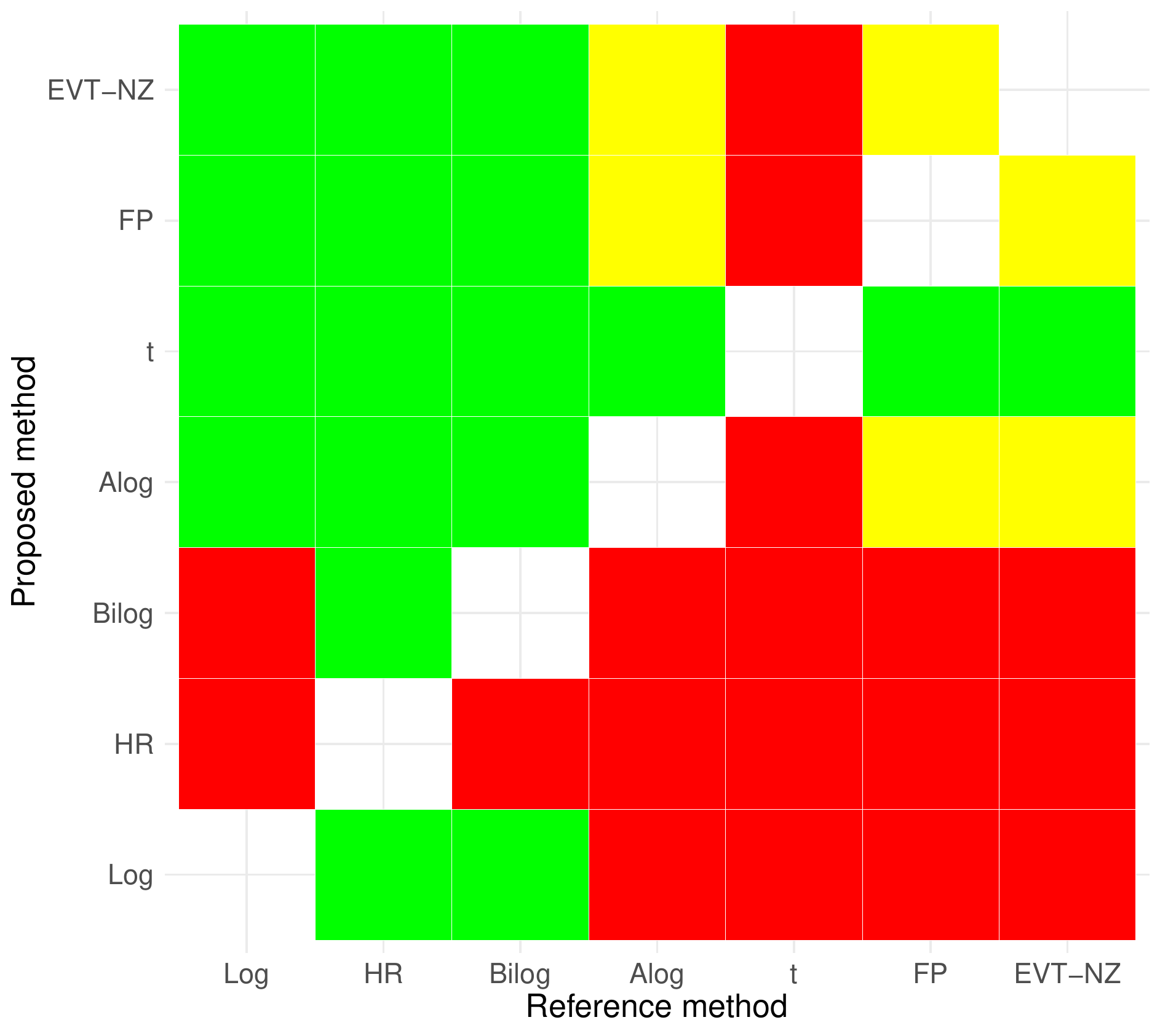}}
\subfigure[Out-of-sample]{\includegraphics[width=0.32\linewidth]{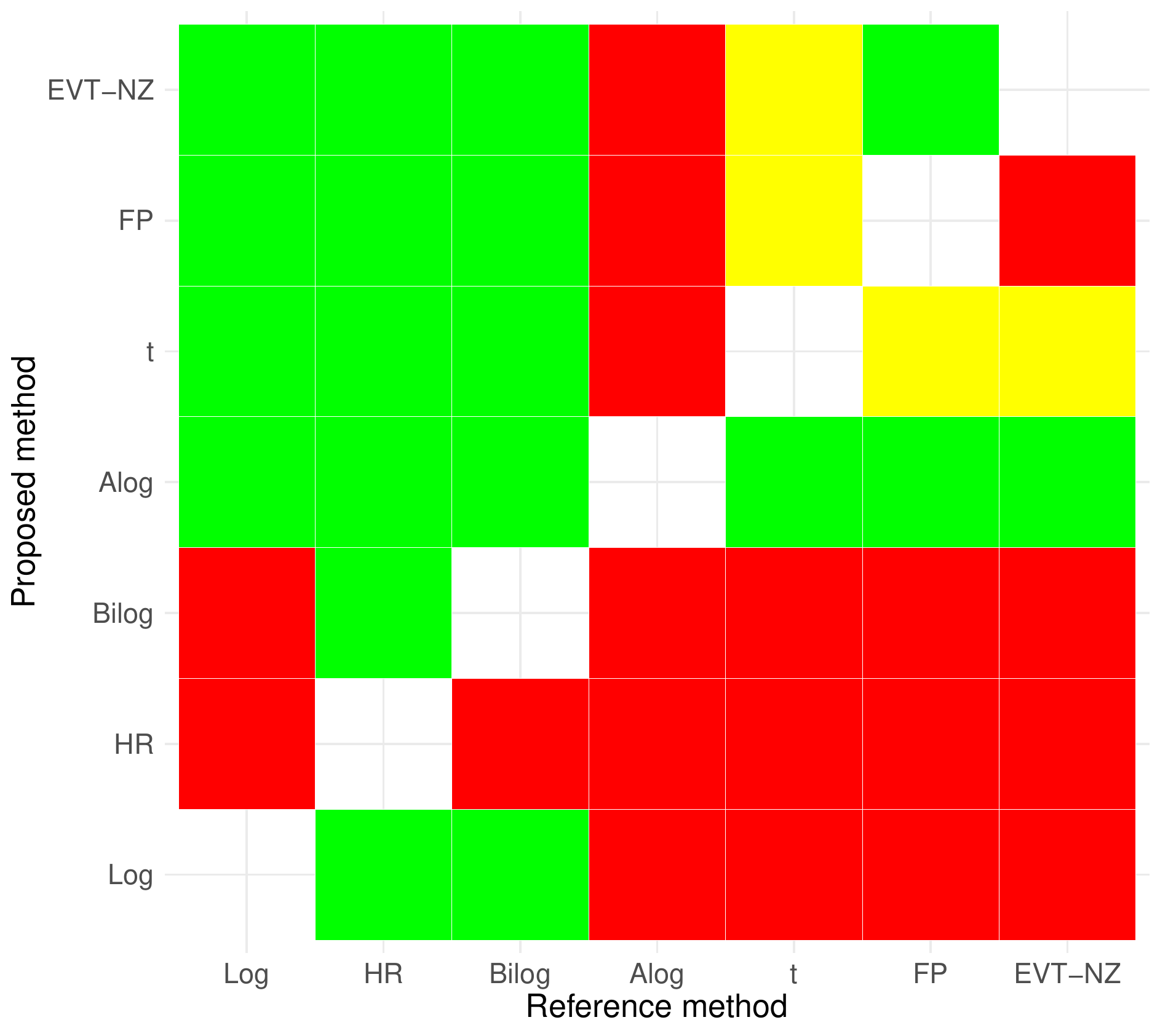}}
\caption{Traffic light matrices for comparative backtests of CoVaR estimates/forecasts at risk level $\pb = (0.02,0.05)$ and confidence test level of 10\%. The left panel is the result of in-sample estimation for institution MBI; the middle and right panels are the results of the in-sample analysis and dynamic out-of-sample forecasts using the combined normalized scores across all considered institutions. Red/green indicates that the reference method is significantly better/worse in terms of the forecasting accuracy than the proposed method. Yellow indicates that the score differences are not statistically significant. }
\label{figTLM}
\end{figure}

\section{Conclusion}\label{sconc}
The paper develops an EVT-based semi-parametric method for estimating the CoVaR, a predominant systemic risk measure. The methodology rests on the existence of a non-degenerate tail dependence function: with modelling the tail dependence function parametrically, we address the data sparsity issue in the joint tail regions. The eventual estimator follows the nonparametric extrapolation techniques in univariate tail estimation. The proposed CoVaR estimator is shown to be consistent. Simulation studies illustrate good performance of the estimator and indicate that its bias and variance are dominated by that of the tail index estimator. Using time series data for 14 financial institutions, we find that the proposed method provides a highly competitive alternative to other existing approaches, while allowing for more flexible model assumptions.

We note that parametric modelling of the tail dependence function comes with the challenge of model selection as well as computational complexity. The latter is especially an issue when the dimension of the parameter vector of the selected model is large and the tail dependence function does not have a simple explicit form to carry through the M-estimation. Finding a more flexible yet computationally tractable way to model tail dependence may help to improve the current framework. Another limitation of the proposed methodology is that it can only be applied to situations in which the tail dependence function is non-degenerate, i.e., in the presence of tail dependence. While tail dependence is a reasonable and most relevant assumption for many financial time series, in some situations, especially in the context of environmental applications, extending the current methodology to also include the case of tail independence will be useful. 

\section*{Supplementary Materials}
In the supplementary appendix, Section S1 provides the proof of Theorem~\ref{main theorem}. Section S2 provides (upper) tail dependence functions for the five parametric models used in simulation studies and the application. Section S3 contains plots of $R(1,\eta)$ as a function of $\eta$ under several tail dependence models. Section S4 summarizes the two-stage procedure for producing dynamic CoVaR forecasts.

\if0\blind
{
\section*{Acknowledgements} N. Nolde and M. Zhou acknowledge financial support of the UBC-Scotiabank Risk Analytics Initiative and Natural Sciences and Engineering Research Council of Canada. 
} \fi

\newpage
\begin{landscape}
\begin{table}[H]
\tiny
\centering
\caption{Unconditional coverage tests for VaR of institutions and CoVaR based on realized residuals at level $\pb = (0.02,0.05)$.}
\vspace{12pt}
\begin{tabular}{cc|cccccccccccccc}
\hline
                        &          & AFL    & AIG    & ALL    & BAC    & HUM             & JPM    & LNC    & MBI             & PGR    & SLM             & TRV    & UNM             & WFC    & WM              \\ \hline
\multirow{4}{*}{VaR}    & estimate & 2.1794 & 2.2086 & 2.1058 & 2.1984 & 2.0821          & 2.1272 & 2.2398 & 2.1250          & 2.1367 & 2.1036          & 2.1784 & 2.2624          & 2.1453 & 2.1878          \\
                        & $E_n$    & 114    & 115    & 110    & 112    & 113             & 117    & 112    & 107             & 112    & 114             & 111    & 116             & 116    & 118             \\ 
                        & $e_n$    & 110.68 & 110.68 & 110.68 & 110.68 & 110.68          & 110.68 & 110.68 & 110.68          & 110.68 & 110.68          & 110.68 & 110.68          & 110.68 & 110.68          \\
                        & p-value  & 0.7511 & 0.6802 & 0.9479 & 0.8993 & 0.8243          & 0.5476 & 0.8993 & 0.7224          & 0.8993 & 0.7511          & 0.9755 & 0.6122          & 0.6122 & 0.4868          \\ \hline
\multirow{4}{*}{Log}    & estimate & 4.8831 & 4.9201 & 4.7874 & 5.0023 & 4.3055          & 5.0284 & 5.0233 & 4.6398          & 4.7298 & 4.6643          & 4.8385 & 4.9313          & 4.9520 & 4.7716          \\
                        & $E_n^b$  & 3      & 3      & 4      & 3      & 4               & 3      & 3      & 1               & 5      & 2               & 5      & 1               & 3      & 2               \\
                        & $e_n^b$  & 5.7    & 5.75   & 5.5    & 5.6    & 5.65            & 5.85   & 5.6    & 5.35            & 5.6    & 5.7             & 5.55   & 5.8             & 5.8    & 5.9             \\
                        & p-value  & 0.2037 & 0.1969 & 0.4912 & 0.2179 & 0.4533          & 0.1839 & 0.2179 & \textbf{0.0187} & 0.7912 & 0.0678          & 0.8077 & \textbf{0.0121} & 0.1903 & 0.0575          \\ \hline
\multirow{4}{*}{HR}     & estimate & 5.0669 & 5.0809 & 5.046  & 5.1109 & 4.7667          & 5.1203 & 5.1189 & 4.9573          & 4.9857 & 4.9450          & 5.0537 & 5.0878          & 5.0961 & 5.0247          \\
                        & $E_n^b$  & 3      & 3      & 3      & 3      & 1               & 3      & 3      & 1               & 4      & 1               & 4      & 1               & 3      & 1               \\
                        & $e_n^b$  & 5.7    & 5.75   & 5.5    & 5.6    & 5.65            & 5.85   & 5.6    & 5.35            & 5.6    & 5.7             & 5.55   & 5.8             & 5.8    & 5.9             \\
                        & p-value  & 0.2037 & 0.1969 & 0.2330 & 0.2179 & \textbf{0.0140} & 0.1839 & 0.2179 & \textbf{0.0187} & 0.4657 & \textbf{0.0134} & 0.4783 & \textbf{0.0121} & 0.1903 & \textbf{0.0110} \\ \hline
\multirow{4}{*}{Bilog}  & estimate & 4.9124 & 4.9310 & 4.7970 & 5.0073 & 4.3330          & 5.0311 & 5.0197 & 4.6578          & 4.7344 & 4.6786          & 4.8472 & 5.1222          & 4.9577 & 4.7895          \\
                        & $E_n^b$  & 3      & 3      & 4      & 3      & 4               & 3      & 3      & 1               & 5      & 2               & 5      & 1               & 3      & 2               \\
                        & $e_n^b$  & 5.7    & 5.75   & 5.5    & 5.6    & 5.65            & 5.85   & 5.6    & 5.35            & 5.6    & 5.7             & 5.55   & 5.8             & 5.8    & 5.9             \\
                        & p-value  & 0.2037 & 0.1969 & 0.4912 & 0.2179 & 0.4533          & 0.1839 & 0.2179 & \textbf{0.0187} & 0.7912 & 0.0678          & 0.8077 & \textbf{0.0121} & 0.1903 & 0.0575          \\ \hline
\multirow{4}{*}{Alog}   & estimate & 4.5123 & 4.5334 & 4.4339 & 4.6032 & 3.9927          & 4.6147 & 4.6163 & 4.2892          & 4.3817 & 4.2994          & 4.4702 & 4.5337          & 4.5585 & 4.4075          \\
                        & $E_n^b$  & 5      & 4      & 5      & 6      & 4               & 6      & 6      & 2               & 5      & 4               & 6      & 4               & 6      & 2               \\
                        & $e_n^b$  & 5.7    & 5.75   & 5.5    & 5.6    & 5.65            & 5.85   & 5.6    & 5.35            & 5.6    & 5.7             & 5.55   & 5.8             & 5.8    & 5.9             \\
                        & p-value  & 0.7589 & 0.4293 & 0.8243 & 0.8638 & 0.4533          & 0.9495 & 0.8638 & 0.0901          & 0.7912 & 0.4412          & 0.8465 & 0.4177          & 0.9325 & 0.0575          \\ \hline
\multirow{4}{*}{t}      & estimate & 4.4518 & 4.4776 & 4.2584 & 4.7650 & 3.6674          & 4.8043 & 4.7302 & 3.9946          & 4.3163 & 4.1515          & 4.3661 & 4.5109          & 4.716  & 4.1211          \\
                        & $E_n^b$  & 5      & 4      & 7      & 5      & 6               & 4      & 5      & 3               & 6      & 4               & 8      & 4               & 5      & 4               \\
                        & $e_n^b$  & 5.7    & 5.75   & 5.5    & 5.6    & 5.65            & 5.85   & 5.6    & 5.35            & 5.6    & 5.7             & 5.55   & 5.8             & 5.8    & 5.9             \\
                        & p-value  & 0.7589 & 0.4293 & 0.5282 & 0.7912 & 0.8811          & 0.4064 & 0.7912 & 0.2573          & 0.8638 & 0.4412          & 0.3155 & 0.4177          & 0.7273 & 0.3952          \\ \hline
\multirow{4}{*}{FP}     & estimate & 4.7033 & 4.5165 & 4.4981 & 4.2482 & 4.0225          & 4.3067 & 4.3846 & 4.3765          & 4.2854 & 4.156           & 4.2367 & 4.7418          & 4.1679 & 4.2795          \\
                        & $E_n^b$  & 4      & 4      & 5      & 8      & 4               & 9      & 6      & 2               & 6      & 4               & 8      & 3               & 7      & 3               \\
                        & $e_n^b$  & 5.7    & 5.75   & 5.5    & 5.6    & 5.65            & 5.85   & 5.6    & 5.35            & 5.6    & 5.7             & 5.55   & 5.8             & 5.8    & 5.9             \\
                        & p-value  & 0.4412 & 0.4293 & 0.8243 & 0.3268 & 0.4533          & 0.2140 & 0.8638 & 0.0901          & 0.8638 & 0.4412          & 0.3155 & 0.1903          & 0.6200 & 0.1777          \\ \hline
\multirow{4}{*}{EVT-NZ} & estimate & 4.2783 & 4.2868 & 4.1591 & 4.2683 & 3.6413          & 4.046  & 4.2945 & 4.3109          & 4.2885 & 4.2570          & 4.0336 & 4.6699          & 4.1215 & 4.3309          \\
                        & $E_n^b$  & 5      & 5      & 7      & 8      & 6               & 10     & 9      & 2               & 6      & 4               & 8      & 3               & 7      & 3               \\
                        & $e_n^b$  & 5.7    & 5.75   & 5.5    & 5.6    & 5.65            & 5.85   & 5.6    & 5.35            & 5.6    & 5.7             & 5.55   & 5.8             & 5.8    & 5.9             \\
                        & p-value  & 0.7589 & 0.7430 & 0.5282 & 0.3268 & 0.8811          & 0.1082 & 0.1738 & 0.0901          & 0.8638 & 0.4412          & 0.3155 & 0.1903          & 0.6200 & 0.1777        \\ \hline 
\end{tabular}
\label{inno_uncon}
\end{table}
\end{landscape}

\begin{table}[H]
\centering
\scriptsize
\caption{The average quantile scores of CoVaR estimates at level $\pb = (0.02,0.05)$. The top panel gives the results of CoVaR estimates based on the whole dataset; the bottom panel gives the results of dynamic CoVaR forecasts. The smallest score for each company is highlighted in boldface.}
\vspace{24pt}
\begin{tabular}{ccccccccc}
\hline                                                                           &     & Log    & HR     & Bilog  & Alog   & t      & FP     & EVT-NZ \\
 \hline 
 \multicolumn{1}{c|}{\multirow{14}{*}{\begin{tabular}[c]{@{}c@{}}In-sample\\ Analysis\end{tabular}}} & AFL & 0.2826 & 0.2870 & 0.2833 & 0.2784 & 0.2780 & 0.2799 & {\bf 0.2770}  \\
\multicolumn{1}{c|}{} & AIG & 0.2832 & 0.2870 & 0.2834 & 0.2751 & 0.2742 & 0.2748 & {\bf 0.2721}  \\
\multicolumn{1}{c|}{} & ALL & 0.2702 & 0.2752 & 0.2703 & {\bf 0.2675} & 0.2689 & 0.2678 & 0.2702  \\
\multicolumn{1}{c|}{}& BAC & 0.2861 & 0.2886 & 0.2862 & {\bf 0.2813} & 0.2816 & 0.2848 & 0.2844  \\
\multicolumn{1}{c|}{}& HUM & 0.2248 & 0.2393 & 0.2252 & 0.2203 & {\bf 0.2191} & 0.2207 & 0.2192  \\
\multicolumn{1}{c|}{}& JPM & 0.2852 & 0.2874 & 0.2852 & {\bf 0.2791} & 0.2804 & 0.2812 & 0.2887  \\
\multicolumn{1}{c|}{}& LNC & 0.2866 & 0.2888 & 0.2865 & {\bf 0.2812} & 0.2814 & 0.2821 & 0.2846  \\
\multicolumn{1}{c|}{}& MBI & 0.2378 & 0.2507 & 0.2386 & 0.2244 & {\bf 0.2162} & 0.2272 & 0.2251  \\
\multicolumn{1}{c|}{}& PGR & 0.2911 & 0.2934 & 0.2911 & {\bf 0.2892} & 0.2894 & 0.2895 & 0.2895  \\
\multicolumn{1}{c|}{}& SLM & 0.2394 & 0.2501 & 0.2399 & 0.2289 & {\bf 0.2266} & 0.2267 & 0.2282  \\
\multicolumn{1}{c|}{}& TRV & {\bf 0.2921} & 0.2948 & 0.2922 &  0.2930 & 0.2936 & 0.2965 & 0.3010  \\
\multicolumn{1}{c|}{}& UNM & 0.2494 & 0.2559 & 0.2573 & 0.2391 & {\bf 0.2387} & 0.2430 & 0.2413  \\
\multicolumn{1}{c|}{}& WFC & 0.2776 & 0.2811 & 0.2778 & {\bf 0.2736} & 0.2738 & 0.2761 & 0.2766  \\
\multicolumn{1}{c|}{}& WM  & 0.2486 & 0.2582 & 0.2492 & 0.2366 & {\bf 0.2304} & 0.2332 & 0.2345 \\ \hline \hline

\multicolumn{1}{c|}{\multirow{14}{*}{\begin{tabular}[c]{@{}c@{}}Dynamic\\ Forecasting\end{tabular}}}  & AFL & 0.2671 & 0.2735 & 0.2714 & {\bf 0.2582} & 0.2583 & 0.2655 & 0.2628  \\
\multicolumn{1}{c|}{}& AIG & 0.2653 & 0.2705 & 0.2659 & {\bf 0.2577} & 0.2581 & 0.2643 & 0.2624  \\
\multicolumn{1}{c|}{}& ALL & 0.3010 & 0.3083 & 0.3023 & 0.2943 & 0.2982 & {\bf 0.2942} & 0.2992  \\
\multicolumn{1}{c|}{}& BAC & 0.2411 & 0.2449 & 0.2415 & {\bf 0.2295} & 0.2317 & 0.2536 & 0.2329  \\
\multicolumn{1}{c|}{}& HUM & 0.2415 & 0.2424 & {\bf 0.2404} & 0.2480 & 0.2862 & 0.2457 & 0.2619  \\
\multicolumn{1}{c|}{}& JPM & 0.2452 & 0.2473 & 0.2453 & {\bf 0.2360} & 0.2392 & 0.2535 & 0.2445  \\
\multicolumn{1}{c|}{}& LNC & 0.2270 & 0.2300 & 0.2272 & 0.2135 & 0.2150 & 0.2205 & {\bf 0.2117}  \\
\multicolumn{1}{c|}{}& MBI & 0.3052 & 0.3263 & 0.3069 & 0.2864 & {\bf 0.2764} & 0.2927 & 0.2799  \\
\multicolumn{1}{c|}{}& PGR & 0.2741 & 0.2800 & 0.2743 & {\bf 0.2741} & 0.2759 & 0.2742 & 0.2765  \\
\multicolumn{1}{c|}{}& SLM & 0.2352 & 0.2466 & 0.2358 & {\bf 0.2245} & 0.2295 & 0.2274 & 0.2281  \\
\multicolumn{1}{c|}{}& TRV & 0.2708 & 0.2788 & 0.2710 & {\bf 0.2683} & 0.2708 & 0.2690 & 0.2726  \\
\multicolumn{1}{c|}{}& UNM & 0.2359 & 0.2405 & 0.2365 & 0.2219 & {\bf 0.2209} & 0.2288 & 0.2218  \\
\multicolumn{1}{c|}{}& WFC & 0.2355 & 0.2390 & 0.2355 & 0.2241 & 0.2270 & 0.2242 & {\bf 0.2197}  \\
\multicolumn{1}{c|}{}& WM  & 0.2554 & 0.2707 & 0.2580 & 0.2485 & 0.2521 & 0.2495 & {\bf 0.2479} \\ \hline
\end{tabular}
\label{ave_score}
\end{table}

\bibliographystyle{plainnat} 

\clearpage
\appendix
\def\thesection{S\arabic{section}}
\def\thesubsection{\arabic{section}\arabic{subsection}}
\setcounter{equation}{0}
\renewcommand{\theequation}{S\arabic{equation}}
\setcounter{page}{1}

\begin{NoHyper}

\begin{center}
    {\LARGE An Extreme Value Approach to CoVaR Estimation\\
  \bf Supplementary Appendix\\
  }
\end{center}

\section{Proof of Theorem \ref{main theorem}}
\label{sup:proof}
\label{appen_proof}
Recall the definition of $\eta_p$:
$$\h_p = \dfrac{\pbb\bigl(Y\ge \CoVaR_{Y|X}(p)\bigr)}{\pbb\bigl(Y\ge \CoVaR_{Y|X}(p)\mid X\ge \VaR_X(p)\bigr)}\in(0,1],\qquad p\in(0,1).$$
We can then rewrite the ratio $\frac{\widehat{\CoVaR}_{Y|X}(p)}{\CoVaR_{Y|X}(p)}$ as 
\begin{align*}
    \frac{\widehat{\CoVaR}_{Y|X}(p)}{\CoVaR_{Y|X}(p)} &=\frac{\left(\hat\eta_p^*\right)^{-\hat\gamma}\widehat{\VaR}_Y(p)}{\VaR_Y(p\eta_p)}\\
    &=\left(\frac{\hat\eta_p^*}{\eta_p}\right)^{-\hat\gamma}\times \eta_p^{\gamma-\hat\gamma
    } \times \frac{\widehat{\VaR}_Y(p)}{\VaR_Y(p)} \times\frac{\left(\eta_p\right)^{-\gamma}\VaR_Y(p)}{\VaR_Y(p\eta_p)} \\
    &=:I_1\times I_2\times I_3\times I_4.
\end{align*}
The theorem is proved by showing that, as $n\to\infty$, $I_j\top1$ for $j=1,2,3,4$. 

Before handling these four terms, the following two lemmas provide some preliminary results regarding the quantities $\eta_p$ and $\eta^*_p$ as well as the estimator $\hat\eta^*_p$.

\begin{lemma} \label{lem:eta* and hat eta*}
    Under the same conditions as in Theorem \ref{main theorem}, we have that, as $n\to\infty$,
$$\frac{p}{\eta^*_p}\to R_2(1,0;\thb_0) \text{\ and \ }\frac{p}{\hat\eta^*_p}\top R_2(1,0;\thb_0).$$
\end{lemma}
\begin{lemma} \label{lem:eta and eta*}
    Under the same conditions as in Theorem \ref{main theorem}, we have, as $n\to\infty$,
    $$\frac{\eta_p}{\eta^*_p}\to1.$$
\end{lemma}
\begin{proof}[Proof of Lemma \ref{lem:eta* and hat eta*}]
In order to prove the limit relation regarding $\eta^*_p$, we first show that as $n\to\infty$, $\eta^*_p\to 0$. If otherwise, then there exists a subsequence of integers, $\{n_l\}$ such that $\eta^*_{p(n_l)}\to c>0$ as $l\to\infty$. W.l.o.g., we still use the notation $n$ instead of $n_l$. Then, as $n\to\infty$, $R(1,\eta^*_p;\thb_0)\to R(1,c;\thb_0)>0$ which follows from \ref{con:R2} and the fact that $R(1,y;\thb_0)$ is a non-decreasing function in $y$. However, this contradicts with $R(1,\eta^*_p;\thb_0)=p\to 0$ as $n\to\infty$. Hence, we conclude that $\eta^*_p\to 0$ as $n\to\infty$.

Using the mean value theorem, we have that there exists a series of constants $\xi_n\in[0, \eta^*_p]$ such that
$$
p=R(1,\eta^*_p;\thb_0)=R(1,0;\thb_0)+\eta^*_p R_2(1,\xi_n;\thb_0)=\eta^*_p R_2(1,\xi_n;\thb_0).
$$
Hence we get that, as $n\to\infty$,
$$\frac{p}{\eta^*_p}=R_2(1,\xi_n;\thb_0)\to R_2(1,0;\thb_0).$$
Here in the last step, we use the fact that $\xi_n\to 0$ as $n\to\infty$ and $R_2(x,y;\thb_0)$ is a continuous function at $(1,0;\thb_0)$.

The proof for the limit relation regarding $\hat\eta^*_p$ follows similarly by replacing $\thb_0$ with $\hat\thb$ and using the fact that $\hat\thb\top\thb_0$ as $n\to\infty$. We therefore omit the details.
\end{proof}

\begin{proof}[Proof of Lemma \ref{lem:eta and eta*}]
    We first show that, as $n\to\infty$, $\eta_p\to 0$. If assuming otherwise, there exists a subsequence of integers, $\{n_l\}$ such that $\eta_{p(n_l)}\to c>0$ as $l\to\infty$. W.l.o.g., we still use the notation $n$ instead of $n_l$. Recall the definition of $\eta_p$:
    $$\frac{\pbb(X>\VaR_X(p),Y>\VaR_Y(p\eta_p))}{p}=p.$$
    By taking $n\to\infty$ on both sides of this equation, and using the assumption that $\eta_p\to c>0$ as $n\to\infty$, we get that
    $R(1,c;\thb_0)=0$, which contradicts \ref{con:R2} and the fact that $R(1,y;\thb_0)$ is a non-decreasing function in $y$. Hence, we conclude that, as $n\to\infty$, $\eta_p\to 0$.

Next we show, by contradiction, that 
$$\limsup_{n\to\infty}\frac{\eta_p}{\eta^*_p}\leq 1.$$
If assuming otherwise, there exists a subsequence of $n$, $\{n_l\}_{l=1}^\infty$ such that as $l\to\infty$, $n_l\to\infty$ and
$$\frac{\eta_{p(n_l)}}{\eta^*_{p(n_l)}}\to c>1.$$
W.l.o.g., we still use the notation $n$ for the subsequence, and omit it by writing $p=p(n)$. Therefore, for any $1<\tilde c<c$, there exists $n_0=n_0(\tilde c)$ such that for $n>n_0$, $\frac{\eta_p}{\eta^*_p}>\tilde c.$

Note that $\eta_p>\tilde c \eta^*_p>\eta^*_p$. By the mean value theorem, we get that for each $n$, there exists $\xi_n\in(\eta^*_p,\eta_p)$ such that 
$$R(1,\eta_p;\thb_0) -R(1,\eta^*_p;\thb_0) =R_2(1,\xi_n;\thb_0)(\eta_p-\eta^*_p).$$
As $n\to\infty$, since both $\eta^*_p\to 0$ and $\eta_p\to 0$ hold, we get $\xi_n\to 0$. Further note that $\eta_p-\eta^*_p>(\tilde c-1)\eta^*_p$. By applying Lemma \ref{lem:eta* and hat eta*} and the continuity of $R_2(x,y;\thb)$ at $(1,0;\thb_0)$, we get that 
\begin{align*}
\liminf_{n\to\infty}\frac{R(1,\eta_p;\thb_0) -p}{p}&=\liminf_{n\to\infty}\frac{R(1,\eta_p;\thb_0) -R(1,\eta^*_p;\thb_0)}{p}\\
&=\liminf_{n\to\infty}\frac{R(1,\eta_p;\thb_0) -R(1,\eta^*_p;\thb_0)}{\eta^*_p}\times\frac{\eta^*_p}{p}\\
&\geq R_2(1,0;\thb_0)(\tilde c-1)\times \frac{1}{R_2(1,0;\thb_0)}=\tilde c-1>0.
\end{align*}
Since \ref{con:soc for R} holds with $\tilde\rho>1$, we get that
$$\lim_{n\to\infty}\frac{R(1,\eta_p;\thb_0) -p}{p}=\lim_{n\to\infty}\frac{1}{p}\left(R(1,\eta_p;\thb_0)-\frac{1}{p}\pbb(X>\VaR_X(p),Y>\VaR_Y(p\eta_p))\right)=0.$$
The two limit relations contradict each other. Therefore, we conclude that $$\limsup_{n\to\infty}\frac{\eta_p}{\eta^*_p}\leq 1.$$
Similarly, one can show a lower bound for $\frac{\eta_p}{\eta^*_p}$, which completes the proof of the lemma.
\end{proof}

Now we turn to prove the main theorem by handling the four terms $I_j$, $j=1,2,3,4$.

Firstly, we handle $I_1$. Following the asymptotic property of the Hill estimator (e.g., Theorem 3.2.5 in \cite{dHF2006_sup}), \ref{con:soc} and \ref{con:k} for $k_1$ imply that as $n\to\infty$, 
\begin{equation}\label{eq:asymptotic Hill}
    \sqrt{k_1}(\hat\gamma-\gamma)\stackrel{d}{\to} N\left(\frac{\lambda_1}{1-\rho},\gamma^2\right),
\end{equation}
which implies that
$\hat\gamma\top \gamma$. Together with Lemma \ref{lem:eta and eta*} and Lemma \ref{lem:eta* and hat eta*}, we conclude that $I_1\top 1$ as $n\to\infty$.

Secondly, we handle $I_2$. Given the limit relation in \eqref{eq:asymptotic Hill}, we only need to show that $\log (\eta_p)/\sqrt{k_1}\to 0$ as $n\to\infty$. From Lemma \ref{lem:eta* and hat eta*} and Lemma \ref{lem:eta and eta*}, we get that $\eta_p/p\to 1/R_2(1,0,\thb_0)$ as $n\to\infty$. Together with the limit relation regarding $k_1$ in \ref{con:k}, we get that $I_2\top 1$ as $n\to\infty$.

The term $I_3$ is handled by the asymptotic property of the VaR estimator; see, e.g. Theorem 4.3.8 in \cite{dHF2006_sup}. More specifically, under \ref{con:soc}, \ref{con:k} and \ref{con:p}, the VaR estimator in Section~\ref{method:est} has the following asymptotic property: as $n\to\infty$,
$$ \min\left(\sqrt{k_2},\frac{\sqrt{k_1}}{\log (k_2/np)}\right)\left(\frac{\widehat{\VaR}_Y(p)}{\VaR_Y(p)}-1\right)=O_P(1). $$
The result follows from the proof of Theorem 4.3.8 in \cite{dHF2006_sup} with some proper adaptations. A direct consequence is that $I_3\top1$ as $n\to\infty$.

Finally, we handle the deterministic term $I_4$. Notice that $\VaR_Y(p)=U_Y(1/p)$ and $\VaR_Y(p\eta_p)=U_Y(1/(p\eta_p))$. By applying \ref{con:soc} with $t=1/p$ and $x=1/\eta_p$, we get that
$$\lim_{n\to\infty}\frac{\frac{\VaR_Y(p\eta_p)}{\VaR_Y(p)}\eta_p^\gamma-1}{A(1/p)}=-\frac{1}{\rho}.$$
As $n\to\infty$, since $A(1/p)\to0$ we get that $I_4\to 1$.
\qed

\section{Parametric models for the tail dependence function}\label{sS3}
In this section, we provide tail dependence functions for five parametric models considered in simulation studies and the application. 
\begin{enumerate}[label=(\arabic*)]
    \item The bivariate logistic distribution function with standard Fr\'echet margins is given by
    $$G(x,y;\theta) = \exp\left\{-(x^{-1/\theta}+y^{-1/\theta})^\theta\right\},$$
    where $x, y>0$ and $\theta \in (0,1]$. The upper tail dependence function in this case has the form
\begin{equation}\label{tail_log}
R(x,y;\theta) = x + y - (x^{1/\theta}+y^{1/\theta})^{\theta}.
\end{equation}
    
    \item The bivariate H\"{u}sler-Reiss distribution function with standard Fr\'echet margins is
    $$G(x,y;\theta) = \exp\left\{-x^{-1}\Phi\Bigl(\theta^{-1}+\frac{\theta}{2}\log(y/x)\Bigr)-y^{-1}\Phi\Bigl(\theta^{-1}+\frac{\theta}{2}\log(x/y)\Bigr)\right\},$$
    where $x, y>0$, $\theta >0$ and $\Phi(\cdot)$ is the standard normal distribution function. Its tail dependence function is given by
\begin{equation}\label{tail_hr}
R(x,y;\theta) = x + y - x\Phi\Bigl(\theta^{-1}+\frac{\theta}{2}\log(x/y)\Bigr)-y\Phi\Bigl(\theta^{-1}+\frac{\theta}{2}\log(y/x)\Bigr).
\end{equation}

    \item The bilogistic distribution function with standard Fr\'echet margins is given by
$$
 G(x,y;\alpha,\beta) = \exp\left\{-x^{-1}q^{1-\alpha}-y^{-1}(1-q)^{1-\beta}\right\},\quad x,y>0,
$$
where $q$ is the root of the equation $(1-\alpha)x^{-1}(1-q)^{\beta}-(1-\beta)y^{-1}q^{\alpha} = 0$, and $0<\alpha,\beta<1$. The tail dependence function of this distribution can be written as
\begin{equation}\label{tail_bilog}
R(x,y;\alpha,\beta) = x + y - \int_0^1 \max\left\{(1-\alpha)t^{-\alpha}x,(1-\beta)(1-t)^{-\beta}y \right\}dt.
\end{equation}

    \item The bivariate asymmetric logistic distribution with standard Fr\'echet margins has distribution function of the form
$$
 G(x,y;\psi_1,\psi_2,\theta) = \exp\Bigl\{-(1-\psi_1)/x-(1-\psi_2)/y- \bigl((\psi_1/x)^{1/\theta}+(\psi_2/y)^{1/\theta}\bigr)^\theta\Bigr\},
$$
where $x,y>0$, $\theta\in(0,1]$ and $\psi_1,\psi_2\in [0,1]$. Its tail dependence function is given by
\begin{equation}\label{tail_alog}
R(x,y;\psi_1,\psi_2,\theta)=\psi_1x+\psi_2y-\bigl((x\psi_1)^{1/\theta}+(y\psi_2)^{1/\theta}\bigr)^{\theta}.
\end{equation}
    
    \item The joint density function of a standard bivariate $t$ distribution with $\n>0$ degrees of freedom and correlation parameter $\r\in(-1,1)$ is written as
$$
 f_{T}(\wb;\rho,\nu) = \dfrac{\Gamma\bigl((\nu+2)/2\bigr)}{\sqrt{1-\rho^2}\nu\pi\Gamma(\nu/2)}\Bigl(1+\dfrac{1}{\nu}\wb^T\Omega^{-1}\wb\Bigr)^{-(\nu+2)/2},\quad \Omega=\bma 1 & \r\\ \r & 1 \ema,\quad \wb\in\rbb^2.
 $$
For $\r\in(0,1)$, its upper tail dependence function is given by
\begin{equation}\label{tail_t}R(x,y;\rho,\nu)=xF_{T}\Bigl(\sqrt{\frac{\nu+1}{1-\rho^2}}\bigl(\rho-(y/x)^{-1/\nu}\bigr);\nu+1 \Bigr)+yF_{T}\Bigl(\sqrt{\frac{\nu+1}{1-\rho^2}}\bigl(\rho-(x/y)^{-1/\nu}\bigr);\nu+1 \Bigr), 
\end{equation}
where $F_T(\cdot;\n)$ is the distribution function of the standard Student $t$ distribution with $\n$ degrees of freedom. Expression~\eqref{tail_t} is founded following \cite{DemartaMcNeil2005}, where the lower tail dependence function of the bivariate $t$ distribution is given.

\end{enumerate}

\section{Plots of \texorpdfstring{$R(1,\h)$}{R1eta} as a function of \texorpdfstring{$\h$}{eta}}
\label{sup:tail_fun}

Figure~\ref{tail_fun} illustrates function  $R(1,\h)$ under the five tail dependence models listed in Section~\ref{sS3} and used in simulation studies and the application. The upper bound $R(1,\h)\le \h$ for $\h\in(0,1)$ corresponds to the case of complete positive dependence. This implies that when tail dependence is fairly strong, $R(1,\h)$ is close to the linear function $R(1,\h)=\h$. For a given $p$, $\h^*$ in $R(1,\eta^*) = p$ decreases as tail dependence gets stronger. 

\begin{figure}[ht]
  \centering 
  \subfigure[Logistic]{
    \includegraphics[width=0.3\linewidth]{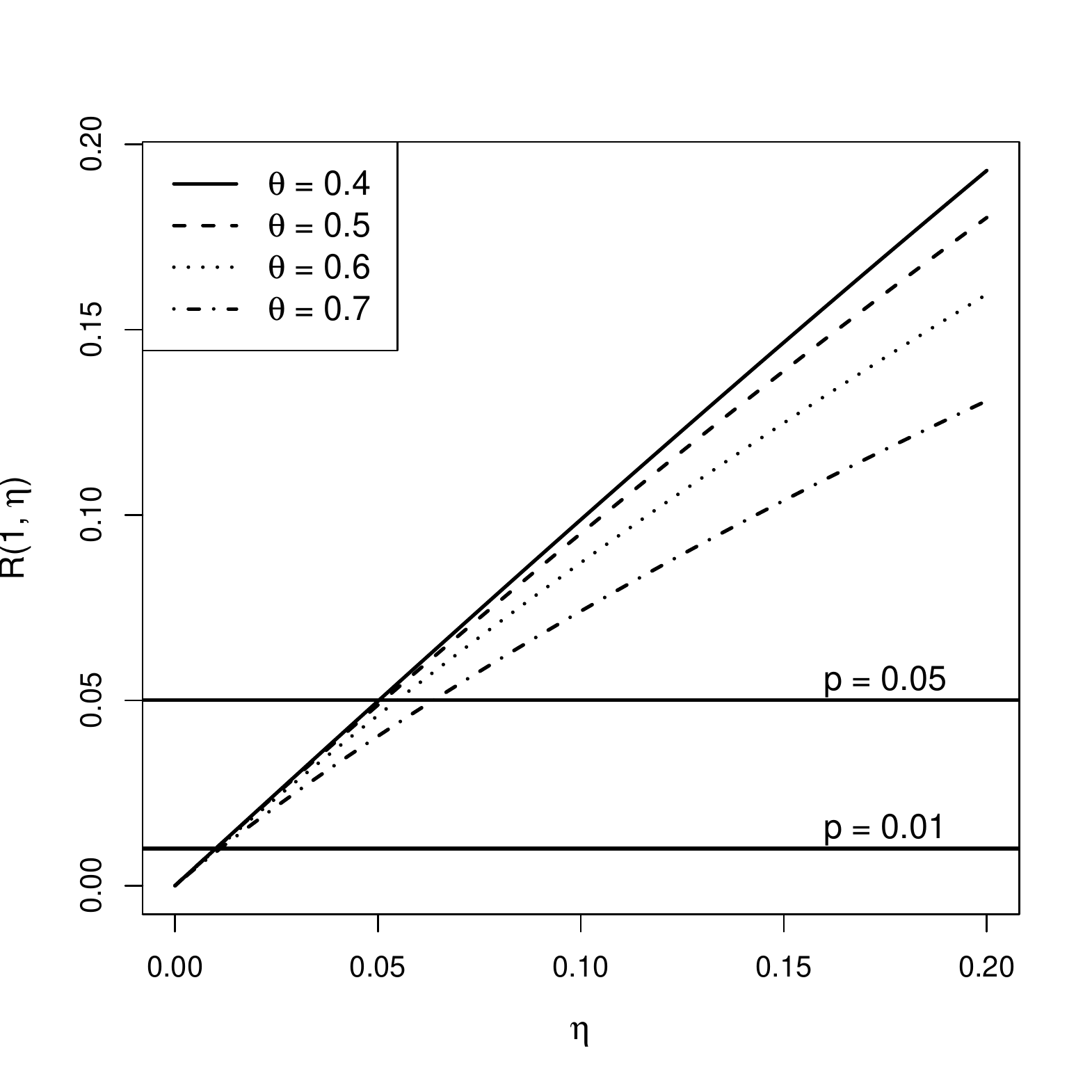}}
  \subfigure[H\"{u}sler-Reiss]{
    \includegraphics[width=0.3\linewidth]{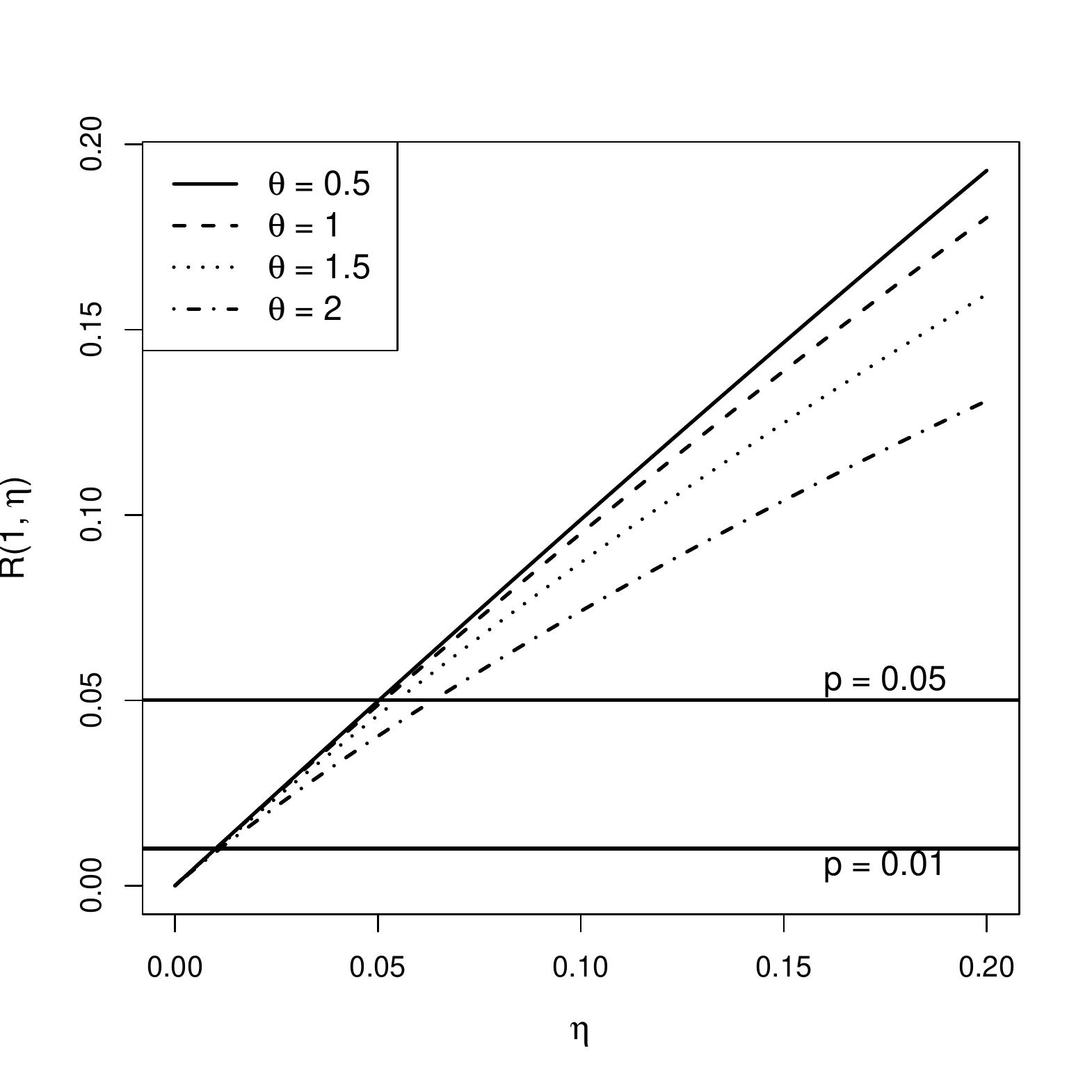}}
    \subfigure[Bilogistic]{
    \includegraphics[width=0.3\linewidth]{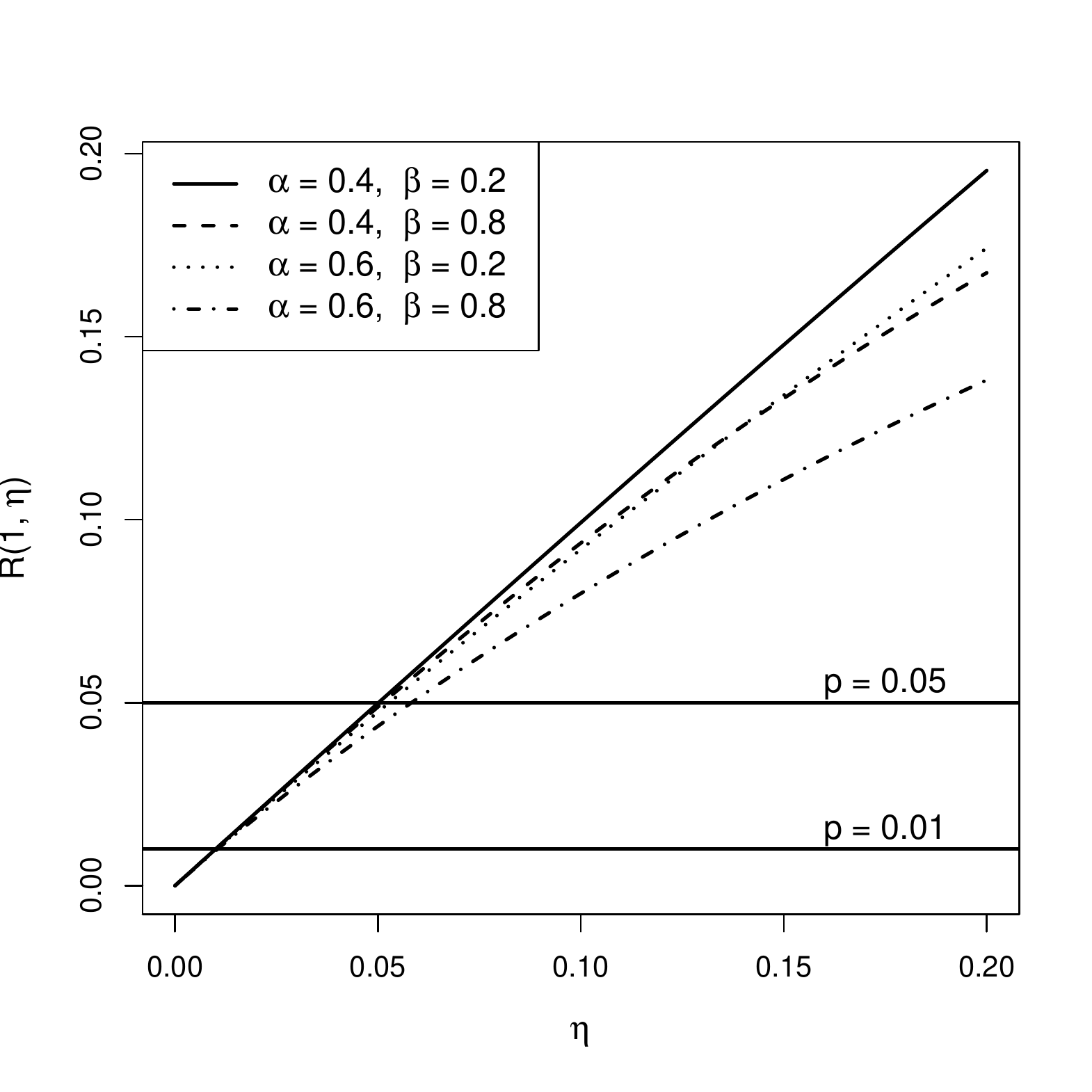}}
    \subfigure[Asymmetric logistic]{
    \includegraphics[width=0.3\linewidth]{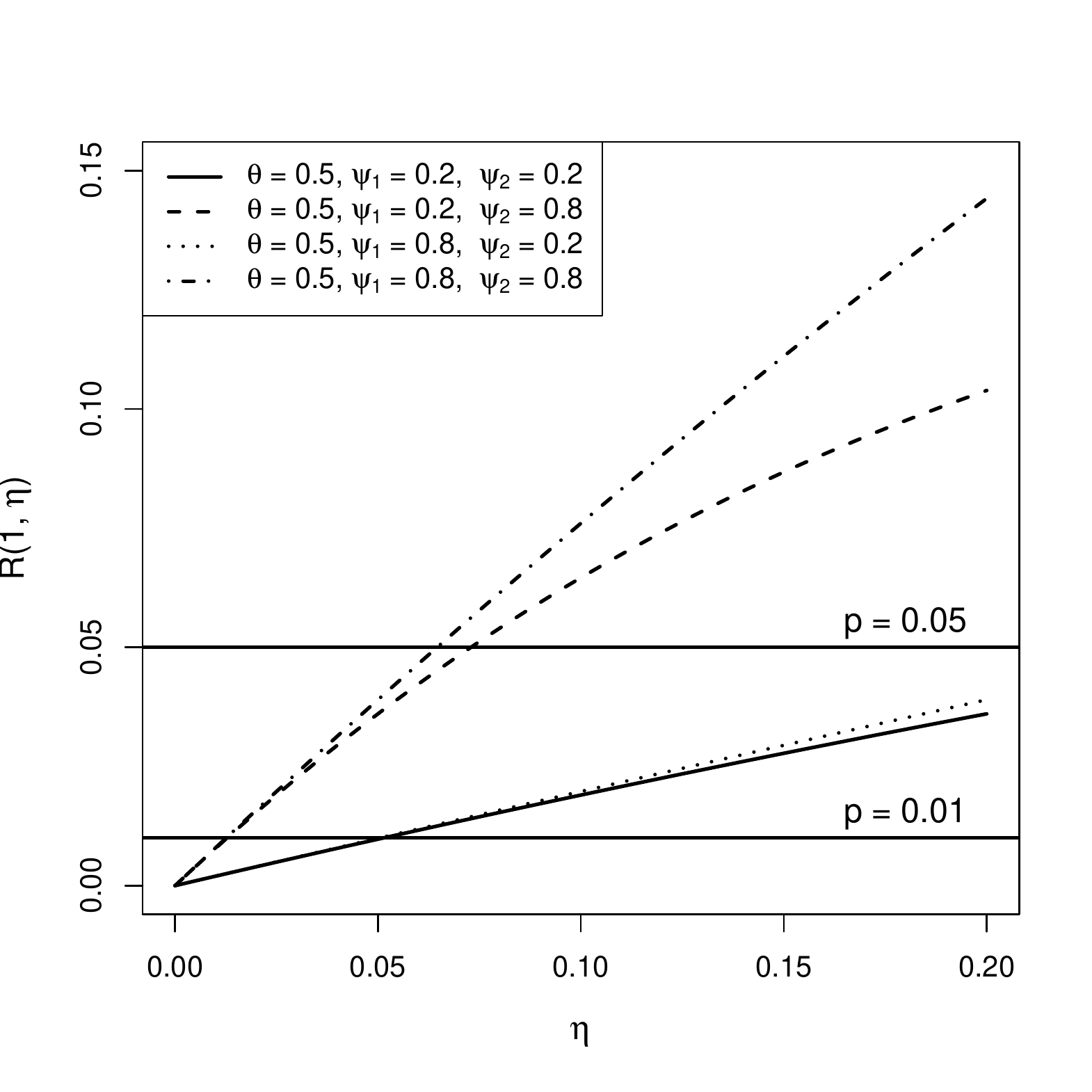}}
    \subfigure[t]{
    \includegraphics[width=0.3\linewidth]{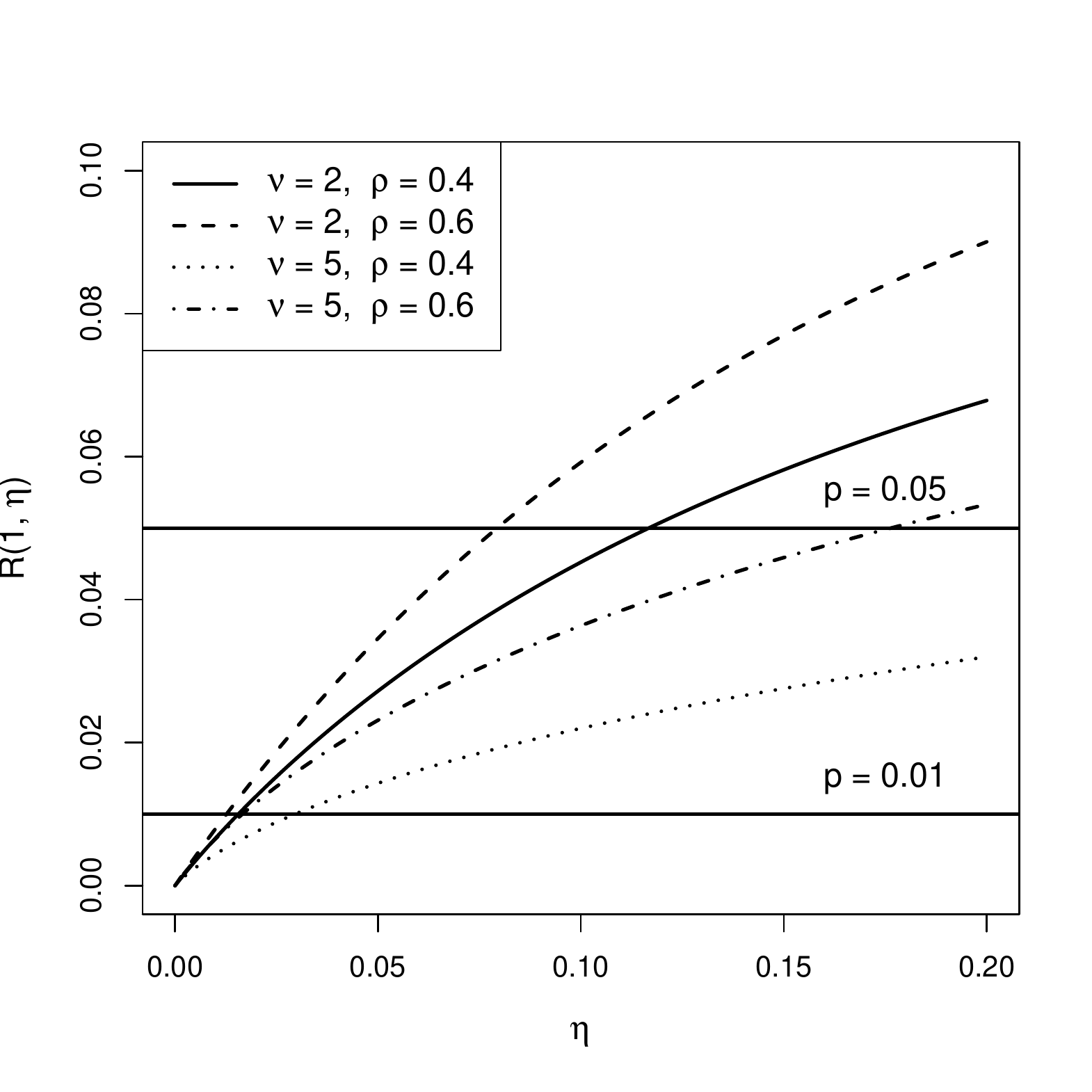}}
    \caption{Plots of $R(1,\eta)$ as a function of $\eta$ under five tail dependence models.}
  \label{tail_fun}
\end{figure}

\section{Two-stage procedure for estimating \texorpdfstring{$\CoVaR_{t}^{s|i}(p_1,p_2)$}{CoVaRt}}\label{sS4}

To estimate CoVaR dynamically, we need to apply a two-stage procedure, which can be summarized in the steps below.

\noindent{\bf Step 1 }(Univariate GARCH model estimation): Assume that $\{X_t^i\}_{t\in \nbb}$ and $\{X_t^s\}_{t\in \nbb}$ each follows an AR(1)-GARCH(1,1) process \citep{Bollerslev1986} satisfying the following model equations:
\begin{align*}
&X_t^{i} = \mu_t^{i} + \sigma_t^{i}Z_t^{i}, \quad \mu_t^{i} = \alpha_0^{i} + \alpha_1^{i}X_{t-1}^{i}, \quad (\sigma_t^{i})^2 = \beta_0^{i} + \beta_1^{i}(\sigma_{t-1}^{i}Z_{t-1}^{i})^2 + \beta_2^{i}(\sigma_{t-1}^{i})^2,\\
&X_t^{s} = \mu_t^{s} + \sigma_t^{s}Z_t^{s}, \quad \mu_t^{s} = \alpha_0^{s} + \alpha_1^{s}X_{t-1}^{s}, \quad (\sigma_t^{s})^2 = \beta_0^{s} + \beta_1^{s}(\sigma_{t-1}^{s}Z_{t-1}^{s})^2 + \beta_2^{s}(\sigma_{t-1}^{s})^2,
\end{align*}
where sequences of innovations $\{Z_t^{i}\}_{t\in \nbb}$ and $\{Z_t^{s}\}_{t\in \nbb}$ are  i.i.d. with zero mean and unit variance. Parameters of the AR(1)-GARCH(1,1) filters are estimated using maximum likelihood assuming a standardized skew-$t$ distribution (\cite{FernandezSteel1998}) for the innovations. With the estimates of conditional means and volatilities, we can obtain two sequences that could be used as proxies for realized innovations: 
\begin{equation}
\label{inno}
    \bigl\{\hat{Z}_t^{i} = (X_t^{i} - \hat{\mu}_t^{i})/{\hat{\sigma}_t^{i}}\bigr\},\qquad \bigl\{\hat{Z}_t^{s} = (X_t^{s} - \hat{\mu}_t^{s})/{\hat{\sigma}_t^{s}}\bigr\}.
\end{equation}

\noindent{\bf Step 2 }(Dynamic CoVaR estimation): Based on the time series representation of losses, $\CoVaR_{t}^{s|i}(p_1,p_2)$ can be expressed as
\begin{align*}
1 - p_2  &= \pbb\bigl(X_t^s \geq \CoVaR_{t}^{s|i}(p_1,p_2) | X_t^i\geq \VaR^i_{t}(p_1);\ \FC_{t-1}^i,\   \FC_{t-1}^s\bigr)\\
&=\pbb\Bigg(Z_t^s \geq \dfrac{\CoVaR_{t}^{s|i}(p_1,p_2)-\m_t^s}{\s_t^s}\Big | Z_t^i\geq \dfrac{\VaR^i_{t}(p_1)-\m_t^i}{\s_t^i};\ \FC_{t-1}^i,\   \FC_{t-1}^s\Biggr).
\end{align*}
This suggests first estimating risk measures based on the samples of realized innovations in \eqref{inno}, treated as i.i.d., and then computing the dynamic forecasts for time $t$ via
\begin{equation}
\widehat{\CoVaR}_{t}^{s|i}(p_1,p_2) = \hat{\mu}_t^{s} + \hat{\s}_t^s \widehat{\CoVaR}_{Z^s|Z^i}(p_1,p_2),\qquad \widehat{\VaR}^i_{t}(p_1)=\hat{\mu}_t^i + \hat{\s}_t^i\widehat{\VaR}_{Z^i}(p_1).
\end{equation}

We note that in Step~1 above, if there is evidence of time changing correlation structure in the data, an alternative is to use a bivariate GARCH filter as was previously done in \cite{Girardi2013_sup} and \cite{NoldeZhang2018_sup}. For the data considered here, filtering out correlation led to weaker tail dependence potentially invalidating the assumption of tail dependence. We, therefore, chose to apply the GARCH filters only marginally.  

\end{NoHyper}

\bibliographystyle{plainnat}


\end{document}